\documentclass[12pt,draftcls,onecolumn]{IEEEtran}

\usepackage{amsmath,amssymb,enumerate,color,bbm,graphicx}
\usepackage[latin1]{inputenc}
\usepackage[english]{babel}
\usepackage{graphicx}
\usepackage{multirow}

\definecolor{gris75}{gray}{0.3}
\newtheorem{remark}{Remark}

\newtheorem{theo}{Theorem}
\newtheorem{lemma}{Lemma}

\newtheorem{assum}{Assumption}

\newcommand{\un}{{\bs 1}}

\newcommand{\eqdef}{:=}

\newcommand{\PP}[1]{{\mathbb{P}}\left\{ #1 \right\}}

\newcommand{\bs}{\boldsymbol}

\newcommand{\sW}{{\mathsf W}}

\newcommand{\sE}{{\mathsf E}}
\newcommand{\sV}{{\mathsf V}}
\newcommand{\sd}{{\mathsf d}}

\newcommand{\cL}{{\mathcal L}}

\newcommand{\cF}{{\mathcal F}}
\newcommand{\cN}{{\mathcal N}}

\newcommand{\cO}{{\mathcal O}}

\newcommand{\bP}{{\mathbb P}}
\newcommand{\bR}{{\mathbb R}}

\newcommand{\bE}{{\mathbb E}}

\newcommand{\la}{\langle}
\newcommand{\ra}{\rangle}

\newcommand{\bth}{{\bs \theta}}
\newcommand{\obth}{{\bs \theta}_{\bot}}
\newcommand{\bY}{{\bs Y}}

\newcommand{\ath}{\langle \bth\rangle}
\newcommand{\thn}{{\bs \theta}_{n}}
\newcommand{\othn}{{\bs \theta}_{\bot,n}}
\newcommand{\bthn}{\bar {\bs \theta}_{n}}
\newcommand{\thk}{{\bs \theta}_{k}}

\newcommand{\thnmu}{{\bs \theta}_{n-1}}
\newcommand{\othnmu}{{\bs \theta}_{\bot,n-1}}

\newcommand{\athn}{\langle\thn\rangle}

\newcommand{\athnmu}{\langle\thnmu\rangle}

\newcommand{\Jo}{{J_\bot}}

\newcommand{\Yon}{\bs Y_{\!\!\!\bot,n}}
\newcommand{\Yok}{\bs Y_{\!\!\!\bot,k}}

\newcommand{\bsY}{{\boldsymbol Y}}
\newcommand{\bsYn}{{\boldsymbol Y}_n}

\newcommand{\dlim}{\ensuremath{\stackrel{\mathcal{D}}{\longrightarrow}}}

\title{Performance of a Distributed Stochastic Approximation Algorithm}

\author{ Pascal Bianchi,~\IEEEmembership{Member,~IEEE}, Gersende Fort, Walid  Hachem,~\IEEEmembership{Member,~IEEE} \thanks{The authors are  with LTCI -
    CNRS/TELECOM ParisTech, 46 rue Barrault, 75634 Paris Cedex 13, France
    (e-mail: $\{$name$\}$@telecom-paristech.fr).  This work is partially
    supported by the French National Research Agency under the program ANR-07
    ROBO 002.}  }

\begin{document}

\maketitle

\begin{abstract}
 In this paper, a distributed stochastic approximation algorithm is
studied.  Applications of such algorithms include decentralized
estimation, optimization, control or computing. The algorithm consists
in two steps: a local step, where each node in a network updates a
local estimate using a stochastic approximation algorithm with
decreasing step size, and a gossip step, where a node computes a local
weighted average between its estimates and those of its
neighbors. Convergence of the estimates toward a consensus is
established under weak assumptions. The approach relies on two main
ingredients: the existence of a Lyapunov function for the mean field
in the agreement subspace, and a contraction property of the random
matrices of weights in the subspace orthogonal to the agreement
subspace.  A second order analysis of the algorithm is also performed
under the form of a Central Limit Theorem. The Polyak-averaged version
of the algorithm is also considered.
\end{abstract}

\section{Introduction}
\label{sec:intro} 
Stochastic approximation has been a very active research area for the last
sixty years (see e.g. \cite{benveniste:metivier:priouret:1987,kushner:2003}).
The pattern for a stochastic approximation algorithm is provided by the
recursion $\theta_{n} = \theta_{n-1} + \gamma_{n} Y_{n}$, where $\theta_n$ is
typically a $\bR^d$-valued sequence of parameters, $Y_n$ is a sequence of
random observations, and $\gamma_n$ is a deterministic sequence of step sizes.
An archetypal example of such algorithms is provided by stochastic gradient
algorithms. These are characterized by the fact that $Y_n = - \nabla
g(\theta_{n-1}) + \xi_n$ where $\nabla g$ is the gradient of a function $g$ 
to be minimized, and where $(\xi_n)_{n \geq 0}$ is a noise sequence corrupting 
the observations.
\\
In the traditional setting, sensing and processing capabilities needed for the
implementation of a stochastic approximation algorithm are centralized on one
machine. Alternatively, distributed versions of these algorithms where the
updates are done by a network of communicating nodes (or agents) have recently
aroused a great deal of interest. Applications include decentralized
estimation, control, optimization, and parallel computing.  

In this paper, we consider a network composed by $N$ nodes (sensors, robots,
computing units, ...). Node $i$ generates a $\bR^d$-valued stochastic process
$(\theta_{n,i})_{n\geq 1}$ through a two-step iterative algorithm:
a local and a so called gossip step. At time $n$:\\
\noindent {\tt [Local step]} Node $i$ generates  a temporary iterate 
$\tilde \theta_{n,i}$ given by 
\begin{equation}
  \label{eq:tempupdate}
\tilde \theta_{n,i}= \theta_{n-1,i} + \gamma_{n}\, Y_{n,i}\ ,
\end{equation}
where $\gamma_n$ is a deterministic positive step size and where the 
$\bR^d$-valued random process $(Y_{n,i})_{n \geq 1}$ represents the 
observations made by agent $i$. 

\noindent {\tt [Gossip step]} Node $i$ is able to observe the values 
$\tilde\theta_{n,j}$ of some other $j$'s and computes the weighted average:
\begin{equation}
  \label{eq:gossipIntro}
\theta_{n,i}=\sum_{j=1}^N w_n(i,j)\,\tilde \theta_{n,j} \ ,
\end{equation}
where the $w_n(i,j)$'s are scalar non-negative random coefficients such that
$\sum_{j=1}^N w_n(i,j)=1$ for any $i$. The sequence of random matrices $W_n
\eqdef [ w_n(i,j) ]_{i,j=1}^N$ represents the time-varying communication
network between the nodes.

{\bf Contributions.} 
This paper studies a distributed stochastic approximation algorithm
in the context of random row-stochastic gossip matrices $W_n$.
\begin{itemize}
\item Under the assumption that the algorithm is stable, we prove convergence
  of the algorithm to the sought consensus.  The unanimous convergence of the
  estimates is also established in the case where the frequency of information
  exchange between the nodes converges to zero at some controlled rate. In
  practice, this means that matrices $W_n$ become more and more likely to be
  equal to identity as $n\to\infty$. The benefits of this possibility in terms
  of power devoted to communications are obvious.
\item We provide verifiable sufficient conditions
for stability. 
\item We establish a Central Limit Theorem (CLT) on the estimates in
  the case where the $W_n$ are doubly stochastic. We show in particular that
  the node estimates tend to fluctuate synchronously for large $n$,
  \emph{i.e.}, the disagreement between the nodes is negligible at the CLT
  scale.  Interestingly, the distributed algorithm under study has the same
  asymptotic variance as its centralized analogue.
\item We also consider a CLT on the sequences averaged over time as
  introduced   in \cite{polyak:1990}. We show that averaging always improves the rate of convergence and the asymptotic variance.
\end{itemize}

\medskip

{\bf Motivations and examples}. The algorithm under study is motivated by the emergence of various decentralized network 
structures such as sensor networks, computer clouds or wireless ad-hoc networks.
One of the main application targets is distributed optimization. In this context, one seeks to minimize a
sum of some local objective differentiable functions $f_i$ of the agents: 
\begin{equation}
\label{eq:distOpt}
{\tt Minimize}\ \  \sum_{i=1}^N F_i(\theta)\ .
\end{equation}
Function $F_i$ is supposed to be unknown by any other agent $j\neq i$. 
In this context, the distributed 
algorithm~(\ref{eq:tempupdate})-(\ref{eq:gossipIntro}) would reduce to a distributed stochastic
gradient algorithm by letting
$Y_{n,i} = -\nabla_\theta F_i(\theta_{n-1,i})+\xi_{n,i}$ where $\nabla_\theta$ is the gradient w.r.t. $\theta$ and
$\xi_{n,i}$ represents some possible random perturbation $\xi_{n,i}$ at time $n$.

In a machine learning context, $F_i$ is typically the risk function of a classifier indexed by $\theta$
and evaluated based on a local training set at agent $i$ \cite{forero-cano-giannakis-jmlr11}.
In a wireless ad-hoc network, $F_i$ represents some (negative) performance measure of a transmission such as 
the Shannon capacity, and the aim is typically to search for a relevant resource allocation vector $\theta$ (see
\cite{bianchi:jakubo:2011} for more details).
As a third example, an application framework to statistical estimation is provided in Section~\ref{sec:appli}. 
In that case, it is assumed that node $i$ receives some independent and 
identically distributed (i.i.d.) time series $(X_{n,i})_n$ 
with probability density function $f_*(x)$.
The system designer considers that the density of $(X_{n,1},\cdots,X_{n,N})$ belongs 
to a parametric family $\{ f(\theta, \bs x) \}_{\theta}$ where
$ f(\theta, \bs x) = \prod_{i=1}^nf_i(\theta,x_i)$.
Then, a well-known contrast for the estimation of $\theta$ is given by the Kullback-Leibler divergence
$D( f_* \, \| \, f(\theta, \cdot) )$ \cite{lehmann-casella:point}. Finding a minimizer boils down to the minimization
of~(\ref{eq:distOpt}) by setting $F_i(\theta) = D( f_{i,*} \, \| \, f_i(\theta, \cdot) )$ where $f_{i,*}$
is the $i$th marginal of $f_{*}$. Then, algorithm~(\ref{eq:tempupdate})-(\ref{eq:gossipIntro}) coincides
with a \emph{distributed online maximum likelihood estimator} by setting 
$Y_{n,i} = -\nabla_\theta \log f_i(\theta_{n-1,i}, X_{n,i})$. Under some regularity conditions, it can be easily checked that
$Y_{n,i} = -\nabla_\theta F_i(\theta_{n-1,i})+\xi_{n,i}$ where $\xi_{n,i}$ is a martingale increment sequence.
\medskip

{\bf Position w.r.t. existing works.} 
There is a rich literature on distributed estimation and optimization 
algorithms, see 
\cite{blo-etal-cdc05},\cite{lop-sayed-asap06}, \cite{kar:2010},                
\cite{cat-sayed-sp10}, \cite{nedic:ozdaglar:parrilo:tac-2010,nedic:tac-2011},  
\cite{stank-stip-tac11} as a non exhaustive list. Among the first gossip       
algorithms are those considered in the treatise \cite{bertsekas:1997} and in   
\cite{tsitsiklis:bertsekas:athans:tac-1986}, as well as in 
\cite{kushner-siam87}, the latter reference dealing with the case of a 
constant step size. The case where the gossip matrices are    
random and the observations are noiseless is considered in \cite{boyd:2006}.   
The authors of \cite{nedic:ozdaglar:parrilo:tac-2010} solve a constrained      
optimization by also using noiseless estimates. The contributions              
\cite{cat-sayed-sp10} and \cite{stank-stip-tac11} consider the framework of    
linear regression models.                                                      

In this paper, the random gossip matrices $W_n$ are assumed to be row
stochastic, \emph{i.e.}, $W_n \un = \un$ where $\un$ is the vector whose components equal one, and column stochastic in the mean,
\emph{i.e.}, $\un^T \bE[W_n] = \un^T$. Observe that the row stochasticity
constraint $W_n \un = \un$ is local, since it simply requires that each agent
makes a weighted sum of the estimates of its neighbors with weights summing to
one. Alternatively, the column stochasticity constraint $\un^T W_n = \un^T$
which is assumed in many contributions (see \emph{e.g.}
\cite{nedic:ozdaglar:tac-2009,nedic:ozdaglar:parrilo:tac-2010,
  ram:nedic:veeravalli:jota-2010,nous:eusipco}) requires a coordination at the
network level (nodes must coordinate their weights). This constraint is not
satisfied by a large class of gossip algorithms. As an example, the well known
broadcast gossip matrices (see Section~\ref{sec:gossip}) are only column
stochastic in the mean.  As opposed to the aforementioned papers, it is worth
noting that some works such as
\cite{kushner-siam87,nedic:tac-2011,bianchi:jakubo:2011} get rid of the
column-stochasticity condition. As a matter of fact, assumption $\un^T\bE[W_n]
= \un^T$ is even relaxed in \cite{kushner-siam87}. Nevertheless, considering
for instance Problem~(\ref{eq:distOpt}), this comes at the price of
losing the convergence to the sought minima. 

In many contributions (see \emph{e.g.} \cite{kushner-siam87},
\cite{lop-sayed-asap06}, or \cite{cat-sayed-sp10}), the gossip step is
performed before the local step, contrary to what is done in this
paper. The general techniques used in this paper to establish the convergence
towards a consensus, the stability and the fluctuations of the estimates
can be adapted without major difficulty to that situation.

In \cite{ram:nedic:veeravalli:jota-2010}, projected stochastic                 
(sub)gradient algorithms are considered in the case where matrices             
$(W_n)_n$ are doubly stochastic. Such results have later been extended by      
\cite{bianchi:jakubo:2011} to the case of non convex optimization, also         
relaxing the doubly-stochastic assumption. It is worth noting that such works  
explicitly or implicitly rely on a projection step onto a compact convex     
set. In many scenarios (such as unconstrained optimization for example), the   
estimate is not naturally supposed to be confined into a known compact set.    
In that case, introducing an artificial projection step is known to modify     
the limit points of the algorithm. On the opposite, this paper addresses the   
issue of unprojected stochastic approximation algorithms. In this context,     
stability turns out to be a crucial issue which is addressed in this paper.    
Note that the stability issues are not considered in most of 
\cite{kushner-siam87}. 
Finally, unlike previous works such as \cite{ram:nedic:veeravalli:jota-2010}   
or \cite{bianchi:jakubo:2011}, we also address the issue of convergence rate   
and characterize the asymptotic fluctuations of the estimation error.          

From a methodological viewpoint, our analysis does not                         
rely on convex optimization tools such as in \emph{e.g.}                       
\cite{nedic:ozdaglar:tac-2009,nedic:ozdaglar:parrilo:tac-2010,                 
ram:nedic:veeravalli:jota-2010}) and does not rely on perturbed differential   
inclusions as in \cite{bianchi:jakubo:2011}. The almost sure convergence       
result is obtained following an approach inspired by \cite{delyon:2000}        
and \cite{andrieu:2005} (other works such as \cite{kushner-siam87} consider 
weak convergence approaches). 
The stability result is obtained by introducing a     
Lyapunov function and by jointly controlling the moments of this Lyapunov      
function and the second order moments of the disagreements between local       
estimates. 
Finally, the study of the asymptotic fluctuations of the estimate   
is based on recent results of \cite{fort:2011} and is partly inspired by the   
works of \cite{pelletier:1998}. 

\medskip

 This paper is organized as follows. In Section \ref{sec:model},
we state and comment our basic assumptions. The algorithm convergence is
studied in Section \ref{sec:analysis}. The second order behavior of the
algorithm is described in Section \ref{sec:rates}.  An application relative to distributed estimation is described in
Section \ref{sec:appli}, along with some numerical simulations.  The appendix is devoted to
the proofs.

\section{The model and the basic assumptions}
\label{sec:model}

Let us start by writing the distributed algorithm described in the previous
section in a more compact form. Define the $\bR^{dN}$-valued random vectors
$\thn$ and $\bsYn$ by $\thn\eqdef(\theta_{n,1}^T,\dots,\theta_{n,N}^T)^T$ and
$\bsYn \eqdef (Y_{n,1}^T,\dots, Y_{n,N}^T)^T$ where $A^T$ denotes the transpose
of the matrix $A$.  The algorithm reduces to:
\begin{equation}
\thn = (W_n\otimes I_d)\left(\thnmu+\gamma_{n} \bsYn\right) \ ,
\label{eq:algo} 
\end{equation}
where $\otimes$ denotes the Kronecker product and $I_d$ is the $d× d$ identity
matrix. 

Note that we always assume $\bE|\bth_0|^2<\infty$ throughout the paper, where $|\,.\,|$
represents the Euclidean norm.
\begin{remark}
  Following~\cite{polyak:1990}, we also consider the {\sl averaged} sequence
  $(\bthn)_{n \geq 1}$, where $\bthn \eqdef(\bar \theta_{n,1}^T,\dots,\bar
  \theta_{n,N}^T)^T$ and the components are given by
\begin{equation}
  \label{eq:algo:aver}
  \bar\theta_{n,i} = \frac 1n\sum_{k=1}^n\theta_{k,i} 
\end{equation}
at any instant $n$ for node $i$. We will show in Section~\ref{sec:TCL} that
this averaging technique improves the convergence rate of the distributed
stochastic approximation algorithm.  In this paper, we analyze the asymptotic
behavior of both sequences $\bthn$ and $\thn$ as $n\to\infty$.
\end{remark}

\subsection{Observation and Network Models}
\label{subsec:model} 
Let $\left(\mu_\bth\right)_{\bth\in \bR^{dN}}$ be a family of probability
measures on $\bR^{dN}$ endowed with its Borel $\sigma$-field
$\mathcal{B}(\bR^{dN})$ such that for any $A \in \mathcal{B}(\bR^{dN})$, $\bth
\mapsto \mu_\bth(A)$ is measurable from $\mathcal{B}(\bR^{dN})$ to
$\mathcal{B}([0,1])$ where $\mathcal{B}([0,1])$ denotes the Borel
$\sigma$-field on $[0,1]$.

We consider the case when the random process $(\bsYn,W_n)_{n\geq 1}$
is adapted to a filtered probability space $\left(\Omega, \mathcal{A}, \bP,
  (\cF_n)_{n \geq 0}\right)$ and satisfy
\begin{assum}
\label{hyp:model}
\begin{enumerate}[a)] 
\item \label{hyp:model:a} $(W_n)_{n\geq 1}$ is a sequence of $N \times N$
  random matrices with non-negative elements such that:
    \begin{itemize}
    \item $W_n$ is row stochastic: $W_n\un=\un$,
    \item $\bE(W_n)$ is column stochastic: $\un^T\bE(W_n)=\un^T$,
    \end{itemize}
\item \label{hyp:model:c} For any positive measurable functions $f,g$ and any $n \geq 0$,
\begin{equation} 
  \bE[f(W_{n+1})g(\bs Y_{n+1})\vert \cF_n] 
=\bE[f(W_{n+1})]\,  \  \int g(\bs y) \mu_{\bs \theta_n}(d \bs y)\ .
\end{equation}
\item \label{hyp:model:iid} The sequence $(W_n)_{n\geq 1}$ is identically
  distributed and the spectral norm $\rho$ of matrix
  $\bE(W_1^T(I_N-\un\un^T/N)W_1)$ satisfies $\rho<1$.
\end{enumerate}
\end{assum}
Assumptions {\bf\ref{hyp:model}\ref{hyp:model:a}}) and
{\bf\ref{hyp:model}\ref{hyp:model:iid}}) capture the properties of the
gossiping scheme within the network.  Following the work of~\cite{boyd:2006},
random gossip is assumed in this paper.  
Assumption {\bf \ref{hyp:model}\ref{hyp:model:a}}) has been commented in 
the introduction. 
The assumption on the spectral norm in
Assumption~{\bf\ref{hyp:model}\ref{hyp:model:iid})} is a connectivity condition
of the underlying network graph which will be discussed in more details in
Section~\ref{sec:gossip}.  Assumption {\bf \ref{hyp:model}\ref{hyp:model:c})}
implies that \textit{(i)} the random variables (r.v.) $W_{n}$ and $\bs Y_{n}$ 
are independent
conditionally to the past, \textit{(ii)} the r.v. $(W_{n})_{n \geq 1}$ are
independent and \textit{(iii)} the conditional distribution of $\bsY_{n+1}$
given the past is $\mu_{\thn}$. This assumption is quite usual in the framework of stochastic approximation and is sometimes refer to
as a Robbins-Monro setting. 
As a particular case, this assumption holds if $\bs Y_{n+1}$ has the form $\bs
Y_{n+1} = g(\bs \theta_n) + \xi_{n+1}$ where $\xi_{n+1}$ is an i.i.d. process.

It is also assumed that the step-size sequence $(\gamma_n)_{n\geq 1}$ in the
stochastic approximation scheme (\ref{eq:tempupdate}) satisfies the following
conditions which are rather usual in the framework of stochastic approximation
algorithms~\cite{kushner:2003}:
\begin{assum} \label{hyp:step}
  The deterministic sequence $(\gamma_n)_{n\geq 1}$ is positive and such that $
  \sum_n\gamma_n=\infty$ and $\sum_n\gamma_n^2<\infty$. 
\end{assum} 


\subsection{Illustration: Some Examples of Gossip Schemes}
\label{sec:gossip}
We describe three standard gossip schemes so called \emph{pairwise},           
\emph{broadcast} and \emph{dropout} schemes. The reader may refer              
to~\cite{benezit:thesis} for a more complete picture and for more general      
gossip strategies. The network of agents is represented as a non-directed      
graph $(\sE,\sV)$ where $\sE$ is the set of edges and $\sV$ is the set of $N$  
vertices.

\subsubsection{Pairwise Gossip}
\label{subsubsec:pairwise} 
This example can be found in~\cite{boyd:2006} on average consensus (see also
\cite{bianchi:jakubo:2011}).

At time $n$, two connected nodes -- say $i$ and $j$ -- wake up, independently
from the past.  Nodes $i$ and $j$ compute the weighted average
$\theta_{n,i}=\theta_{n,j}= 0.5 \tilde\theta_{n,i} + 0.5 \tilde\theta_{n,j}$;
and for $k\notin\{i,j\}$, the nodes do not gossip: $\theta_{n,k}=\tilde
\theta_{n,k}$. In this example, given the edge $\{i,j\}$ wakes up, $W_n$ is
equal to $ I_N-(e_i-e_j)(e_i-e_j)^T/2$ where $e_j$ denotes the $i$th vector of
the canonical basis in $\bR^N$; and the matrices $(W_n)_{n \geq 0}$ are i.i.d.
and doubly stochastic.  Assumption~{\bf
  \ref{hyp:model}\ref{hyp:model:a})} is obviously satisfied.  Conditions for
Assumption~{\bf \ref{hyp:model}\ref{hyp:model:iid})} can be found
in~\cite{boyd:2006}: the spectral norm $\rho$ of the matrix
$\bE(W_n(I_N-\un\un^T/N)W_n^T)$ is in $[0,1)$ if and only if the weighted graph
$(\sE,\sV, \sW)$ is connected, where the wedge $\{i,j\}$ is weighted by the
probability that the nodes $i,j$ communicate.

\subsubsection{Broadcast Gossip}
This example is adapted from the broadcast scheme in~\cite{aysal:2009}.  At
time $n$, a node $i$ wakes up at random with uniform probability and broadcasts
its temporary update $\tilde \theta_{n,i}$ to all its neighbors $\cN_i$. Any
neighbor $j$ computes the weighted average $\theta_{n,j}=
\beta\tilde\theta_{n,i} + (1-\beta) \tilde\theta_{n,j}$. On the other hand, the
nodes $k$ which do not belong to the neighborhood of $i$ (including $i$ itself)
sets $\theta_{n,k}=\tilde\theta_{n,k}$. Note that, as opposed to the pairwise
scheme, the transmitter node $i$ does not expect any feedback from its
neighbors.  Then, given $i$ wakes up, the $(k,\ell)$th component of $W_n$ is
given by:
$$
w_n(k,\ell) = \left\{
  \begin{array}[h]{ll}
    1 & \textrm{if } k\notin \cN_i \textrm{ and }k=\ell \ ,\\
    \beta  & \textrm{if } k\in\cN_i \textrm{ and }\ell=i \ ,\\
    1-\beta & \textrm{if } k\in\cN_i \textrm{ and }k=\ell\ , \\
    0 & \textrm{otherwise.}
  \end{array}\right.
$$
This matrix $W_n$ is not doubly stochastic but $\un^T\bE(W_n)= \un^T$ (see
for instance~\cite{aysal:2009}).  Thus, the matrices $(W_n)_{n\geq 1}$ are
i.i.d. and satisfy the assumption~{\bf \ref{hyp:model}\ref{hyp:model:a})}.  Here
again, it can be shown that the spectral norm $\rho$ of
$\bE(W_n(I_N-\un\un^T/N)W_n^T)$ is in $[0,1)$ if and only if $(\sE,\sV)$ is a
connected graph (see~\cite{aysal:2009}).

\subsubsection{Network dropouts}
In this simple example, the network is subjected from time to time to          
a dropout: consider any sequence of gossip matrices $W_n$ satisfying           
Assumptions~{\bf \ref{hyp:model}-a)} and {\bf \ref{hyp:model}-c)}, and put     
$W'_n = B_n W_n + (1-B_n) I_N$ where $B_n$ is a sequence of i.i.d.~Bernoulli   
random variables independent of the $W_n$. The network 
whose gossip matrices are the $W'_n$ incurs a dropout at the moments where     
$B_n = 0$. At these moments, the nodes locally update their estimates and      
skip the gossip step. It is easy to show that the sequence $W'_n$ satisfies    
Assumptions~{\bf \ref{hyp:model}-a)} and {\bf \ref{hyp:model}-c)}.             

\section{Convergence results}
\label{sec:analysis}
In this section, we address the asymptotic behavior when $n \to \infty$ of the
algorithm (\ref{eq:algo}) and of its averaged version~(\ref{eq:algo:aver}).  To
that goal, we write $\thn$ as the sum of a vector in the consensus space and a
disagreement vector.  Let
\begin{equation}
  \label{eq:projectors}
  J\eqdef (\un\un^T/N)\otimes I_d \ , \qquad \qquad \Jo\eqdef I_{dN}-J \ ,
\end{equation}
be resp.  the projector onto the \emph{consensus subspace} $\left\{ \un\otimes
  \theta : \theta\in\bR^d\right\}$ and the projector onto the orthogonal
subspace.  For any vector ${\bs x} \in\bR^{dN}$, define the vector of $\bR^d$
\begin{equation}
  \label{eq:average}
  \la \bs x \ra \eqdef \frac 1N ({\un^T\otimes I_d})\bs x  \ , 
\end{equation}
so that $J \bs x = \un \otimes \la \bs x \ra$.  Note that $\la \bs x
\ra=(x_1+\dots+x_N)/N$ in case we write $\bs x =(x_1^T,\dots,x_N^T)^T$, $x_i$
in $\bR^d$. Set 
\begin{equation}
  \label{eq:orthogonal}
  \bs x_\bot \eqdef \Jo\bs x
\end{equation}
so that $\bs x=\un\otimes \la \bs x \ra+\bs x_\bot$. We will refer to $\othn
\eqdef \Jo\thn$ as the \emph{disagreement vector}. 

The convergence results rely on the following equations: under Assumption {\bf \ref{hyp:model}\ref{hyp:model:a})}, it holds
\begin{align}
  \athn &= \athnmu + \gamma_n \la (W_n\otimes I_d)(\bsY_n+\gamma_n^{-1}\othnmu)
\ra\ , \label{eq:SAmarkovnoise}\\
  \gamma_{n+1}^{-1} \othn &= \frac{\gamma_n}{\gamma_{n+1}} \Jo (W_n\otimes I_d)
  \left( \gamma_n^{-1} \othnmu + \Jo \bsY_n \right) \ .
\end{align}
We then first address the almost-sure convergence of the sequence $(\thn)_{n
  \geq 1}$ \textit{(i)} by showing that the non-homogeneous controlled Markov
chain $(\gamma_{n-1}^{-1} \othn)_n$ is stable enough so that $(\othn)_n$
converges almost-surely to zero and, \textit{(ii)} by applying results on the
convergence of stochastic approximation algorithms with state-dependent noise
in order to identify the limiting points of the sequence $(\athn)_{n \geq 1}$.
These results are stated in Theorem~\ref{the:cv} (and
Theorem~\ref{the:cv:vanish} in the case of vanishing communication rate); we
prove that all agents eventually reach an agreement on the value of their
estimate: the limit points of $(\thn)_{n \geq 1}$ (resp.  $(\bthn)_{n \geq 1}$)
given by (\ref{eq:algo}) (resp. (\ref{eq:algo:aver})) are of the form
$\un\otimes \theta_\star$.

It is known that convergence of stochastic approximation algorithms to an
attractive set is established provided that the sequence remains in a compact
set with probability one and is, with probability one, infinitely often in the
domain of attraction of this attractive set. Our convergence result is stated
under assumptions implying the recurrence property provided the sequence
remains almost-surely in a compact set. Therefore, our convergence results are
derived under a boundedness assumption, and we then provide in
Theorem~\ref{the:stab} sufficient conditions for this boundedness condition to
be satisfied.

All these convergence results are obtained under conditions on the
state-dependent noise sequence in the stochastic approximation
scheme~(\ref{eq:SAmarkovnoise}). These conditions roughly speaking assume
\textit{(i)} that there exist a Lyapunov function and an attractive set
associated to the mean field of the noisy Ordinary Differential Equation
(\ref{eq:SAmarkovnoise}), \textit{(ii)} regularity-in-$\bth$ of the
probability distributions $(\mu_{\bth})_{\bth \in \bR^{dN}}$. The exact 
assumptions are stated herein. 

\subsection{Assumptions on the distributions $\mu_{\bth}$}
\label{sec:assump:mutheta}
Define the function $h:\bR^d\to\bR^d$ by:
  \begin{equation}
h(\theta) \eqdef    \int  \la \bs y \ra  \, \mu_{\un\otimes \theta}(d \bs y) 
\label{eq:meanfield} \ .
\end{equation}
We shall refer to $h$ as the \emph{mean field}. The key ingredient to prove the
convergence of a stochastic approximation procedure is the existence of a
Lyapunov function $V$ for the mean field $h$ \emph{i.e.}, a function $V: \bR^d
\to \bR^+$ such that $\nabla V^T \ h \leq 0$. Precisely, it is assumed:
\begin{assum} \label{hyp:V} There exists a function $V : \bR^d \to \bR^+$ such 
that:
  \begin{enumerate}[a)]
  \item \label{hyp:V:a} $V$ is continuously differentiable.
  \item \label{hyp:V:b} For any $\theta \in \bR^d$, $ \nabla V(\theta)^T h(\theta) \leq 0$,
    where $h$ is given by (\ref{eq:meanfield}).
\item \label{hyp:V:d} For any $M>0$, the level set $\{\theta\in\bR^d: V(\theta)\leq M \}$ is compact.
\item \label{hyp:V:e} The set $\cL\eqdef \{\theta\in \bR^d: \nabla\!V(\theta)^T
  h(\theta) =0 \}$ is non-empty and there exists $M_0$ such that $\cL \subseteq
  \{V \leq M_0 \}$.
 \item \label{hyp:H:b} The function $h$ given by (\ref{eq:meanfield}) is continuous on $\bR^d$.
\item \label{hyp:V:f} $V(\cL)\eqdef\{V(\theta)\,:\theta\in \cL\}$ has an empty interior.
  \end{enumerate}
\end{assum}
Observe that Assumptions~{\bf \ref{hyp:V}\ref{hyp:V:e})} and ~{\bf \ref{hyp:V}\ref{hyp:V:f})} are trivially 
satisfied when $\cL$ is finite.

When $h$ is a gradient field i.e. $h = - \nabla g$, a natural candidate        
for the Lyapunov function is $V = g$. In this case, $\cL = \{ \nabla g =0 \}$; 
when $g$ is $d$-times differentiable, the Sard's theorem implies that     
$g(\{\nabla g =0\})$ has an empty interior. If $g$ is strictly convex and it   
reaches its minimum at a finite $\theta_\star$, the function 
$\theta \mapsto  | \theta - \theta_\star |^2$ 
is also a Lyapunov function. In this case, $\cL = \{\theta_\star \}$.                                                            

\begin{assum} \label{hyp:H} For any $M>0$,
  \begin{enumerate}[a)]
  \item \label{hyp:H:a}   $\sup_{|\bth|\leq M} \int \left|\bs y\right|^2 \, \mu_{\bth}(d \bs y)<\infty$.
  \item \label{hyp:H:abis} there exists a constant $C_M$ such that for any $|\bth|\leq M$,
\begin{equation}
\left| \int   \la \bs y \ra  \, \mu_{\bth}(d \bs y) -  \int   \la \bs y \ra \, \mu_{\un\otimes \la \bth \ra}(d \bs y)  \right| \leq C_M |\obth|\ .\label{eq:EHLipsch}
\end{equation}
  \end{enumerate}
\end{assum} 
The condition~(\ref{eq:EHLipsch}) is a regularity condition on the distribution
of $\la \bs Y_{n+1} \ra$ given the past. 

\subsection{Almost sure convergence of the distributed algorithm}
\label{sec:cvg}
Define $\sd(\theta,A)\eqdef\inf\{ |\theta-\varphi|\,:\varphi \in A\}$ for any
$\theta\in\bR^d$ and $A\subset\bR^d$.
\begin{theo}
\label{the:cv}
Suppose Assumptions {\bf \ref{hyp:model}}, {\bf \ref{hyp:step}}, {\bf
  \ref{hyp:V}}, {\bf \ref{hyp:H}}. Assume in addition that $\lim_n \gamma_n /
\gamma_{n-1} =1$ and 
\begin{equation}
  \label{eq:assump:stability}
  \PP{ \limsup_n |\thn| < \infty} =1 \ .
\end{equation}
Then, with probability one,
\begin{equation}
\label{eq:cv}
\lim_{n\to\infty}\sd( \la \thn \ra, \cL)=0 \ , \qquad \qquad \lim_n \othn = 0 \ ,
\end{equation}
where $\cL$ is given by Assumption {\bf \ref{hyp:V}}.  Moreover, with probability one,
$(\athn)_{n\geq 1}$ converges to a connected component of $\cL$.
\end{theo}
Theorem~\ref{the:cv} is proved in
Appendix~\ref{sec:proof:the:cv:vanishingrate}.  Theorem~\ref{the:cv} shows that
when the stability condition \eqref{eq:assump:stability} holds true, the vector
of iterates $\thn$ given by (\ref{eq:algo}) converges almost surely to the
consensus space as $n\to\infty$ so that the network asymptotically achieves
consensus. Moreover, this consensus belongs to the attractive set of the
Lyapunov function.

Since $V$ is continuous, Theorem~\ref{the:cv} implies that with probability one
(w.p.1), the sequence $\{V(\la \thn \ra)\}_{n \geq 0}$ converges to a (random)
point $\upsilon_\star \in V(\cL)$. This can be used to show that $(\la \thn \ra
)_{n \geq 0}$ converges to a connected component of $\{\theta \in \cL:
V(\theta) = \upsilon_\star \}$.  In general, this does not imply that $(\la
\thn \ra )_{n \geq 0}$ converges w.p.1 to some (random point) $\theta_\star \in
\cL$. Note nevertheless that this holds true w.p.1 when $\cL$ is finite.

Along any sequence $(\thn)_{n \geq 0}$ converging to $\un \otimes \theta_\star$
for some $\theta_\star \in \cL$, the Cesaro's lemma implies that the averaged
sequence $(\bthn)_{n \geq 0}$ converges w.p.1 to $\un \otimes \theta_\star$.
Therefore, the averaged sequence (\ref{eq:algo:aver}) and the original sequence
(\ref{eq:algo}) have the same limiting value, if any.

\subsection{Case of a vanishing communication rate}
Theorems~\ref{the:cv} still holds true when the r.v. $(W_n)_{n\geq 1}$ are not
identically distributed.  An interesting example is when $\PP{ W_n = I_N } \to
1$ as $n\to \infty$.  From a communication point of view, this means that the
exchange of information between agents becomes rare as $n\to\infty$. This
context is especially interesting in case of wireless networks, where it is
often required to limit as much as possible the amount of communication between
the nodes.

In such cases, Assumption~{\bf\ref{hyp:model}\ref{hyp:model:iid})} does no
longer hold true.  We prove a convergence result for the algorithms
(\ref{eq:algo}) and (\ref{eq:algo:aver}) when the spectral norm of the
matrix $\bE(W_n^T(I_N-\un\un^T/N)W_n)$ and the step size sequence 
$(\gamma_n)_{n \geq 1}$ satisfy the following assumption:
\begin{assum}
\label{hyp:vanish}
$\sum_n \gamma_n = \infty$ and there exists $\alpha>1/2$ such that:
\begin{align}
  &\lim_{n\to\infty} n^{\alpha}\gamma_n = 0  \, ,  \qquad \qquad  \lim_{n \to \infty} n^{1 +\alpha} \gamma_n = +\infty\label{eq:condgamma}\ ,\\
  &\liminf_{n\to\infty} \frac{1-\rho_n}{n^{\alpha}\gamma_n}>0 \ ,
    \label{eq:condrho}
\end{align}
where $\rho_n$ is the spectral norm of the matrix
$\bE(W_n^T(I_N-\un\un^T/N)W_n)$.
\end{assum}
Note that under Assumption {\bf \ref{hyp:vanish}}, $\lim_n n (1-\rho_n) =
+\infty$. A typical framework where this assumption is useful is the following.
Let $(B_n)_n$ be a Bernoulli sequence of independent r.v.  with $\bP(B_n=1)
=p_n$ and the probabilities $p_n$ decrease in such a way that 
$\liminf_n p_n /(n^\alpha \gamma_n) > 0$: replace the matrices
$W_n$ described by Assumption~{\bf \ref{hyp:model}} with $ B_n W_n + (1-B_n)
I_N$.  Here $p_n$ represents the probability that a communication between the
nodes takes place at time $n$.

We also have $\sum_n \gamma_n^2 < \infty$ so that the step-size sequence
$(\gamma_n)_{n \geq 1}$ satisfies the standard conditions for stochastic
approximation scheme to converge.

An example of sequences $(\gamma_n)_{n \geq 1}, (\rho_n)_{n \geq 1}$ satisfying
Assumption~{\bf \ref{hyp:vanish}} is given by $1-\rho_n=a/n^\eta$ and
$\gamma_n=\gamma_0/n^\xi$ with $\eta, \xi$ such that $0\leq \eta < \xi -1/2
\leq 1/2$.  In particular, $\xi\in (1/2,1]$ and $\eta\in [0,1/2)$.

When the r.v.  $(W_n)_{n\geq 1}$ are i.i.d., the spectral norm
$\rho_n$ is equal to $\rho$ for any $n$, and (\ref{eq:condrho}) implies
$\rho<1$: one is back to Assumption~{\bf\ref{hyp:model}\ref{hyp:model:iid})}.
From this point of view, Assumption~{\bf \ref{hyp:vanish}} is weaker than
Assumption~{\bf\ref{hyp:model}\ref{hyp:model:iid})}. Nevertheless, stronger
constraints than Assumption~{\bf\ref{hyp:model}\ref{hyp:model:iid})} are needed
on the step size $(\gamma_n)_{n \geq 1}$.

When substituting Assumption~{\bf\ref{hyp:model}\ref{hyp:model:iid})} by
Assumption~{\bf \ref{hyp:vanish}}, we have
\begin{theo}
\label{the:cv:vanish}
The statement of Theorem~\ref{the:cv} remains valid under Assumptions {\bf
  \ref{hyp:model}\ref{hyp:model:a}-\ref{hyp:model:c})}, {\bf \ref{hyp:step}},
{\bf \ref{hyp:V}}, {\bf \ref{hyp:H}}, {\bf \ref{hyp:vanish}} and
Eq.~(\ref{eq:assump:stability}).
\end{theo}
Theorem~\ref{the:cv:vanish} is proved in Appendix~\ref{sec:proof:the:cv:vanishingrate}.

\subsection{Stability}
In this section, we provide sufficient conditions implying
(\ref{eq:assump:stability}).  These conditions are stated in the case of a
vanishing communication rate but remain valid when Assumption~{\bf
  \ref{hyp:vanish}} is replaced with Assumption~{\bf
  \ref{hyp:model}\ref{hyp:model:iid})}. The proof of Theorem~\ref{the:stab} is
given in Appendix~\ref{sec:proof:the:stab}.

\begin{theo}
  \label{the:stab} Suppose Assumptions {\bf \ref{hyp:model}a-b)}, {\bf \ref{hyp:step}}, {\bf
    \ref{hyp:V}\ref{hyp:V:a}-\ref{hyp:H:b})} and {\bf\ref{hyp:vanish}}. Assume
  in addition that 
 \renewcommand{\labelenumi}{ST\theenumi .}
 \begin{enumerate}
  \item \label{ST:1} $\nabla V$ is Lipschitz on $\bR^d$. 
  \item \label{ST:2} there exists a constant $C$ such that for any $\bs
    \theta \in \bR^{dN}$,
\begin{align*}
& \int \left|\bs y\right|^2 \, \mu_{\bth}(d \bs y)  \leq C\left(1 + V(\ath) + |\obth|^2 \right)  \ , \\
&\left| \int   \la \bs y \ra  \, \mu_{\bth}(d \bs y) -  \int   \la \bs y \ra \, \mu_{\un\otimes \la \bth \ra}(d \bs y)  \right|\leq C |\obth|\ .
\end{align*}
  \end{enumerate}
 Then $\PP{ \limsup_n
    |\thn| < \infty} =1 $.
\end{theo}
It is proved in Appendix~\ref{sec:proof:the:stab} that under the assumptions of
Theorem~\ref{the:stab}, a stronger result holds (see
Lemma~\ref{lem:agreement}): the sequence $(\othn)_{n \geq 1}$ converges to zero
with probability one and $(\athn)_{n \geq 1}$ is stable in the sense that
$\sup_n V(\la \thn \ra) < \infty$. 

Note that the Lipschitz assumption on the gradient $\nabla V$ combined with
Assumption~ST{\bf \ref{ST:2}} implies that $h$ is at most linearly
increasing when $|\bs \theta| \to \infty$.

The stability condition~(\ref{eq:assump:stability}) could also be satisfied by
modifying the algorithm (\ref{eq:algo}) with a truncation step. Truncation on a
fixed compact set of $\bR^{dN}$ is easy to implement and natural when
constraints on the system are available a priori; nevertheless it becomes
impractical and questionable in many situations of interest when a compact set
containing the limiting set $\cL$ is not known a priori. Another stability
strategy consists in truncations on randomly varying compact
sets~\cite{chen:guo:gao:1988}; derivation of conditions implying the stability
of Algorithm~(\ref{eq:algo}) without modifying its limiting set under such an
approach is out of the scope of this paper and left to the interested reader.

\section{Convergence rates}
\label{sec:rates}
In this section, we derive the convergence rate in $L^2$ of the disagreement
sequence $(\othn)_{n }$ defined $\othn \eqdef \Jo \thn$ (see
(\ref{eq:projectors}) and \eqref{eq:orthogonal}). We also derive Central Limit
Theorems for the sequences $(\thn)_n$ and $(\bthn)_n$: we show that averaging
always improves the convergence rate and the asymptotic variance.

\subsection{Convergence rate of the disagreement vector $\othn$}
Whereas Theorem~\ref{the:cv} states that $\lim_n \othn = 0$ almost surely,
Theorem~\ref{the:moment2bounded} provides an information on the convergence
rate: $\othn$ tends to zero in $L^2$ at rate $1/\gamma_n$. For a positive
deterministic sequence $(a_n)_{n\geq 1}$, $\cO(a_n)$ stands for a deterministic
$\bR^\ell$-valued sequence $(x_n)_{n \geq 1}$ such that $\sup_n a_n^{-1}|x_n|
<\infty$. The proof of Theorem~\ref{the:moment2bounded} is given in
Appendix~\ref{sec:proof:the:moment2bounded}.  
\begin{theo}
\label{the:moment2bounded}
Suppose Assumptions {\bf \ref{hyp:model}}, {\bf \ref{hyp:step}} and {\bf \ref{hyp:H}\ref{hyp:H:a})}. For any
$M>0$,
  \begin{equation}
\label{eq:moment2bounded}
  \gamma_n^{-2} \bE \left(|\othn|^2 \un_{\sup_{k \leq n-1} |\bs \theta_k| \leq M}\right) \leq \frac{\rho \ C}{(1-\sqrt\rho)^2}  + \cO\left( \rho^{n} \gamma_n^{-2} \right)
\end{equation}
where $\rho$ is given by Assumption {\bf \ref{hyp:model}\ref{hyp:model:iid})}
and where $C \eqdef \limsup_{n\to\infty} \bE(|\Yon|^2 \un_{\sup_{k \leq n-1} |\bs \theta_k| \leq M})$ is finite.
\end{theo}

\subsection{Central Limit Theorems}
\label{sec:TCL} We derive Central Limit Theorems for sequences $(\thn)_n$ and
$(\bthn)_n$ converging to a point $\un \otimes \theta_\star$ for some
$\theta_\star \in \cL$.  To that goal, we restrict our attention to the case
when the matrix $(W_n)_n$ are doubly stochastic i.e. $\un^T W_n = \un^T$. The
general case is far more technical and out of the scope of this paper. We also
assume that the point $\theta_\star$ and the r.v.  $\bsY$ satisfy
\begin{assum}
  \label{assum:clt:thetastar}
  \begin{enumerate}[a)]
  \item $\theta_\star \in \cL$.
  \item The mean field $h: \bR^d\to \bR^d$ given by (\ref{eq:meanfield}) is
    twice continuously differentiable in a neighborhood of $\theta_\star$.
  \item $\nabla h(\theta_\star)$ is a Hurwitz matrix i.e.  the largest real
    part of its eigenvalues is $-L$ for some $L>0$.
  \end{enumerate}
\end{assum}
\begin{assum}
  \label{assum:clt:Y}
  \begin{enumerate}[a)]
  \item There exist $\delta >0$ and $\tau>0$ such that $\sup_{|\bth - \un
      \otimes \theta_\star | \leq \delta }  \int  |\la \bs y \ra |^{2
        +\tau}  \mu_{\bth}(d \bs y)  < \infty$.
    \item The functions $\bth \mapsto \int \la \bs y \ra \la \bs y \ra^T
      \mu_{\bth}(d \bs y) $ and $\bth \mapsto \int \la \bs y \ra \mu_{\bth}(d
      \bs y) $ are continuous in a neighborhood of $\un \otimes \theta_\star$. 
  \end{enumerate}
\end{assum}
We finally strengthen the assumptions on the step-size sequence $(\gamma_n)_{n
  \geq 0}$.  In the sequel, notations $x_n = o(y_n)$ and $x_n\sim y_n$ stand
for $x_n/y_n\to 0$ and $x_n/y_n\to 1$ respectively.
\begin{assum}
  \label{assum:clt:stepsize}
  \begin{enumerate}[a)]
  \item \label{assum:clt:stepsize:a} $(\gamma_n)_n$ is a positive deterministic
    sequence such that either $\log(\gamma_k/\gamma_{k+1}) = o(\gamma_k)$, or
    $\log(\gamma_k/\gamma_{k+1}) \sim \gamma_k / \gamma_\star$ for some
    $\gamma_\star > 1/(2L)$.
  \item \label{assum:clt:stepsize:c} $\sum_n \gamma_n =  \infty$ and $\sum_n \gamma_n^2 < \infty$.
  \item \label{assum:clt:stepsize:b} $\lim_n n \gamma_n =+\infty$ and
    \begin{gather*}
      \lim_n \frac{1}{\sqrt{n}} \sum_{k=1}^n \gamma_k^{-1/2} \ \left|1 -
  \frac{\gamma_k}{\gamma_{k+1}} \right| =0 \\ \lim_n \frac{1}{\sqrt{n}} \sum_{k=1}^n \gamma_k = 0 \ .
    \end{gather*}
  \end{enumerate}
\end{assum}
The step size $\gamma_n \sim \gamma_\star /n^\xi$ satisfies Assumptions {\bf
  \ref{assum:clt:stepsize}\ref{assum:clt:stepsize:a}-\ref{assum:clt:stepsize:c})}
for any $1/2<\xi\leq 1$ since $\log (\gamma_k / \gamma_{k+1}) \sim \xi/k$.
Similarly, if $\gamma_n \sim \gamma_\star /n$, Assumption {\bf
  \ref{assum:clt:stepsize}\ref{assum:clt:stepsize:a})} holds provided that
$\gamma_\star > (1/2L)$. Observe that when the sequence $(\gamma_n)_n$ is
ultimately non-increasing, then the condition $\lim_n n \gamma_n = +\infty$
implies $\lim_n \sqrt{n}^{-1} \sum_{k=1}^n \gamma_k^{-1/2} \ \left|1 -
  (\gamma_k/\gamma_{k+1}) \right| =0$ (see e.g.~\cite[Theorem 26, Chapter
4]{delyon:2000}).
Set
\begin{multline*}
  \Upsilon \eqdef \int \la \bs y \ra \la \bs y \ra^T \ \mu_{\un
    \otimes \theta_\star}(d \bs y) \\ - \left(\int \la \bs y \ra \
    \mu_{\un \otimes \theta_\star}(d \bs y) \right)\left(\int \la \bs
    y \ra \ \mu_{\un \otimes \theta_\star}(d \bs y) \right)^T \ .
\end{multline*}
\begin{theo}
  \label{th:clt:drm} 
  Suppose Assumptions {\bf \ref{hyp:model}}, {\bf \ref{hyp:H}}, {\bf
    \ref{assum:clt:thetastar}, \ref{assum:clt:Y},
    \ref{assum:clt:stepsize}\ref{assum:clt:stepsize:a}-\ref{assum:clt:stepsize:c})}.
  Assume in addition that $\un^T W_n = \un^T$ w.p.1.  Then under the
  conditional probability $\bP(\cdot \vert \lim_k \thk = \un \otimes
  \theta_\star)$, the sequence of r.v.  $(\gamma_n^{-1/2} \ (\thn - \un \otimes
  \theta_\star))_{n \geq 0}$ converges in distribution to $\un \otimes Z$ where
  $Z$ is a centered Gaussian distribution with covariance matrix $\Sigma$
  solution of the Lyapunov equation:
 $$
 \nabla h(\theta_\star)  \Sigma + \Sigma \nabla h(\theta_\star)^T = -  \Upsilon 
$$
if $\log(\gamma_k /\gamma_{k+1}) = o(\gamma_k)$ and
$$
\left( I + 2 \gamma_\star \nabla h(\theta_\star)  \right) \Sigma + \Sigma \left( I + 2 \gamma_\star \nabla h(\theta_\star)^T \right) = -  \Upsilon
$$
if $\log(\gamma_k /\gamma_{k+1}) \sim \gamma_k/\gamma_\star$.
\end{theo}

\medskip

The proof of Theorem~\ref{th:clt:drm} is postponed to
Appendix~\ref{proof:th:clt:drm}.  The asymptotic variance can be compared to the
asymptotic variance in a centralized algorithm: formally, such an algorithm is
obtained by setting $W_n = \un \un^T/N \otimes I_d$. Interestingly, the
distributed algorithm under study has the same asymptotic variance as its
centralized analogue.

Theorem~\ref{th:clt:drm} shows that when $\gamma_n \sim \gamma_\star /
n^\alpha$ for some $\alpha \in (1/2, 1]$, then the rate in the CLT is
$\cO(1/n^{\alpha/2})$.  Therefore, the maximal rate of convergence is achieved
with $\gamma_n \sim \gamma_\star / n$ and in this case, the rate is
$\cO(1/\sqrt{n})$.  Unfortunately, the use of such a rate necessitates to choose
$\gamma_\star$ as a function of $\nabla h(\theta_\star)$ (through the upper
bound $L$, see Assumption~{\bf
  \ref{assum:clt:stepsize}\ref{assum:clt:stepsize:a})}), and in practice
$\nabla h(\theta_\star)$ is unknown. We will show in
Theorem~\ref{th:clt:drm:aver} that the optimal rate $O(1/\sqrt{n})$ can be
reached by applying the averaged procedure (\ref{eq:algo:aver}) with $\gamma_n
\sim \gamma_\star / n^{\alpha}$ whatever $\alpha \in (1/2, 1)$.

A second question is the scaling of the observations in the local step. Observe
that during each local step of the algorithm (see (\ref{eq:tempupdate})), each
agent can use a common invertible matrix gain $\Gamma$ and update the temporary
iterate $\tilde \theta_{n,i}$ as
\begin{equation}
\label{eq:matricegain}
\tilde \theta_{n,i} = \theta_{n-1,i} + \gamma_{n}\, \Gamma Y_{n,i}\ .
\end{equation}
It is readily seen that the new mean field $\tilde h: \theta \mapsto \int \la
(\Gamma \otimes I_N)\bs y \ra \mu_{\un \otimes \theta}(d \bs y)$ is equal to
$\Gamma h$ and Assumptions~{\bf \ref{hyp:V}} and {\bf \ref{hyp:H}} remain
valid with $(\bsY, h, V)$ replaced by $((\Gamma \otimes I_N) \bsY, \Gamma h,
\Gamma^{-1} V)$.  Therefore, introducing a gain matrix $\Gamma $ does not
change the limiting points of the algorithm~(\ref{eq:algo}) (and thus
(\ref{eq:algo:aver})) but changes the asymptotic variance. In the case of the
optimal rate in Theorem~\ref{th:clt:drm} (i.e. the case $\gamma_n \sim
\gamma_\star /n$ for some $\gamma_\star > 1/(2 L)$), it can be proved following
the same lines as in \cite{fort:2011} (see also \cite[Proposition 4, Chapter 3,
Part I]{benveniste:metivier:priouret:1987}), that the {\em optimal} choice of
the gain matrix is $\Gamma_\star = -\gamma_\star^{-1} \nabla
h(\theta_\star)^{-1}$.  By optimal, we mean that, when weighting the
observations by $\Gamma_\star$ as in~(\ref{eq:matricegain}), the asymptotic
covariance matrix $\Sigma_\star$ obtained through Theorem~\ref{th:clt:drm} is
smaller than the limiting covariance $\Sigma_\Gamma$ associated with any other
gain matrix $\Gamma$ \emph{i.e.}, $\Sigma_\Gamma-\Sigma_\star$ is nonnegative.
Moreover, $\Sigma_\star$ is equal to:
\[
\gamma_\star^{-1} \ \nabla h(\theta_\star)^{-1} \Upsilon \nabla h(\theta_\star)^{-T} \ .
\]
Otherwise stated,  $(\sqrt{n} \ (\athn - \theta_\star))_{n \geq 0}$ converges to 
a centered Gaussian vector with covariance matrix
$\nabla h(\theta_\star)^{-1} \Upsilon \nabla h(\theta_\star)^{-T}$.

In practice, $\nabla h(\theta_\star)$ is unknown and such a
choice of gain matrix cannot be plugged in the algorithm~(\ref{eq:algo}).
Fortunately, Theorem~\ref{th:clt:drm:aver} shows that this optimal variance can be reached
by averaging the sequence $(\bthn)_n$.

Note that these two major features of {\em averaging algorithms} for stochastic approximation
(optimal convergence rate and optimal limiting covariance matrix)
has been pointed out by \cite{polyak:1990} (see also \cite{polyak:juditsky:1992})
in case of centralized algorithms.
\begin{theo}
  \label{th:clt:drm:aver} 
  Let $(\gamma_n)_n$ be a deterministic positive sequence such that $
  \log(\gamma_k/\gamma_{k+1}) = o(\gamma_k)$. Suppose Assumptions {\bf
    \ref{hyp:model}}, {\bf \ref{hyp:H}}, {\bf \ref{assum:clt:thetastar}}, {\bf
    \ref{assum:clt:Y}}, {\bf
    \ref{assum:clt:stepsize}\ref{assum:clt:stepsize:c}-\ref{assum:clt:stepsize:b})}.
  Assume in addition that $\un^T W_n = \un^T$ w.p.1. Then under the conditional
  probability $\bP(\cdot \vert \lim_k \thk = \un \otimes \theta_\star)$, the
  sequence of r.v.  $(\sqrt{n} \ (\bthn - \un \otimes \theta_\star))_{n \geq
    0}$ converges in distribution to $\un \otimes \bar Z$ where $\bar Z$ is a
  centered Gaussian distribution with covariance matrix
  \[
  \nabla h(\theta_\star)^{-1} \ \Upsilon \nabla h(\theta_\star)^{-T} \ .
\]
\end{theo}
The proof of Theorem~\ref{th:clt:drm:aver} is postponed to
Appendix~\ref{proof:th:clt:drm:aver}.

\section{An Application Framework} 
\label{sec:appli}

\subsection{Distributed estimation}
\label{subsec:ml-distributed} 
To illustrate the results, we describe in this section a distributed 
parameter estimation algorithm which converges to a limit point of the 
centralized Maximum Likelihood (ML) estimator. 
Assume that node $i$ receives at time $n$ the $\bR^{m_i}$-valued component 
$X_{n,i}$ of the i.i.d.~random process 
${\bs X}_n = (X_{n,1}^T, \ldots X_{n,N}^T )^T \in \bR^{\sum m_i}$, where 
${\bs X}_1$ has the unknown density $f_*(x)$ with respect to the Lebesgue 
measure. The system designer considers that the density of ${\bs X}_1$ belongs 
to a family $\{ f(\theta, \bs x) \}_{\theta\in \bR^d}$. 
When $f(\theta,\bs x)$ satisfies some regularity and smoothness conditions, 
the limit points of the sequences $\hat\theta_n$ that maximize the 
log-likelihood
function $L_n(\theta) = \sum_{k=1}^n \log f(\theta, \bs X_k)$ are minimizers of
the Kullback-Leibler divergence $D( f_* \, \| \, f(\theta, \cdot) )$ 
\cite{lehmann-casella:point}. Our aim is to design a distributed 
and iterative algorithm that exhibits the same asymptotic behavior in the case 
where $f(\theta,\bs x)$ is of the form 
$f(\theta,\bs x) = \prod_{i=1}^N f_i(\theta, x_i)$ where 
$\bs x  = ( x_1^T, \ldots, x_N^T)^T$ is partitioned similarly to 
${\bs X_1}$.  
To that purpose, Algorithm \eqref{eq:algo} is implemented with the increments
$
Y_{n+1,i} = \nabla_\theta \log f_i\left(\theta_{n,i}, X_{n+1,i}\right) 
$
where $\nabla_\theta$ is the gradient with respect to $\theta$.  
In some sense, $\log f_i(\theta_{n,i}, X_{n+1,i})$ is a local
log-likelihood function that is updated by node $i$ at time $n+1$ by a 
gradient approach. Writing 
${\bs\theta} = ( \theta_1^T, \ldots, \theta_N^T)^T$, the distribution 
$\mu_{\bs\theta}$ introduced in Section \ref{subsec:model} 
is defined by the identity 
\begin{multline*}
  \int g(\bs y) \mu_{\bs \theta}(d \bs y) = \int g\!\big(
  (\nabla_\theta\log f_1(\theta_1, x_1)^T, \dots\ \ \ \ \ \\  \dots,\nabla_\theta\log
  f_N(\theta_N, x_N)^T)^T \big) \ f_*(\bs x) \, d\bs x
\end{multline*}
for every measurable function $g : \bR^{Nd} \to \bR_+$. The associated
mean field given by Equation \eqref{eq:meanfield} will be 
\[
h(\theta) = 
\frac 1N \int \nabla_\theta \log f(\theta, \bs x) \, 
f_*(\bs x) \, d\bs x  . 
\]
Since $h(\theta) = - N^{-1} \nabla_\theta D( f_* \, \| \, f(\theta, \cdot) )$ 
(assuming $\nabla_\theta$ and $\int$ can be interchanged), our algorithm is of
a gradient type with $V(\theta) = D( f_* \, \| \, f(\theta, \cdot) )$ as 
the natural Lyapunov function. Under the assumptions of Theorem \ref{the:cv} 
or Theorem \ref{the:cv:vanish}, we know that the $\theta_{n,i}, i=1,\ldots,N$ 
converge 
unanimously to ${\mathcal L} = \{ \theta \, : \, \nabla V(\theta) = 0 \}$. 
Here, we note that under some weak extra assumptions on the ``noise'' 
of the algorithm, it is possible to show that unstable points such as local 
maxima or saddle points of $V(\theta)$ are avoided (see for instance 
\cite{brandiere-duflo:1996, benaim:cours99, fang-chen:2000}). Consequently,
the first order behavior of the distributed algorithm is identical to that of
the centralized ML algorithm. \\
We now consider the second order behavior of these algorithms, restricting 
ourselves to the case where 
$f_*(\bs x) = \prod_{i=1}^N f_i(\theta_\star, x_i)$ for some 
$\theta_\star \in \bR^d$.
With some conditions on $f_*$, it is well known that any consistent sequence
$\hat\theta_n$ of estimates provided by the centralized ML algorithm satisfies 
$\sqrt{n} ( \hat\theta_n - \theta_\star ) \dlim 
{\mathcal N}(0, F(\theta_\star)^{-1})$ where $\dlim$ stands for the convergence in distribution,
${\mathcal N}(0,\Sigma)$ represents the centered Gaussian distribution with covariance $\Sigma$ and
\[
F(\theta_\star) = \sum_{i=1}^N \int 
\nabla_\theta \log f_i(\theta_\star, x_i) \, 
\nabla_\theta \log f_i(\theta_\star, x_i)^T \, f_i(\theta_\star, x_i) \, dx_i 
\] 
is the Fisher information matrix of $f(\theta_\star, \cdot )$ 
\cite[Chap. 6]{lehmann-casella:point}. We now turn to the distributed 
algorithm and to that end, we apply Theorems \ref{th:clt:drm} 
and \ref{th:clt:drm:aver}. Matrices $\nabla h(\theta_\star)$ and $\Upsilon$ 
found in the statements of these 
theorems coincide in our case with $-N^{-1} F(\theta_\star)$ and 
$N^{-2} F(\theta_\star)$ respectively (same value of $\Upsilon$ for both 
theorems). Starting with the averaged case, Theorem \ref{th:clt:drm:aver} 
shows that on the set 
$\{\lim_n \bth_n = \un \otimes \theta_\star \}$, the averaged sequence 
$\bar{\bth}_n$ satisfies
$\sqrt{n} ( \bar\bth_n - \un \otimes \theta_\star ) \dlim \un \otimes Z$ where 
$Z \sim {\mathcal N}(0, F(\theta_\star)^{-1})$.  
This implies that the averaged algorithm is asymptotically efficient, similarly
to the centralized ML algorithm. Let us consider the non averaged algorithm.
In order to make a fair comparison with the centralized ML algorithm, we 
restrict the use of Theorem \ref{th:clt:drm} to the case where $\gamma_n$ 
has the form $\gamma_n = \gamma_\star/n$. In that case, Assumption 
{\bf \ref{assum:clt:stepsize}} is verified when 
$\gamma_\star > N / (2 \lambda_{\min}(F(\theta_\star)))$ where 
$\lambda_{\min}(F(\theta_\star))$ is the smallest eigenvalue of 
$F(\theta_\star)$. Theorem \ref{th:clt:drm} shows that on the set 
$\{\lim_n \bs\theta_n = \un \otimes \theta_\star \}$, the sequence of 
estimates ${\bth}_n$ satisfies 
$\sqrt{n} ( \bth_n - \un \otimes \theta_\star ) \dlim \un \otimes Z$ where
$Z \sim {\mathcal N}(0, \Sigma)$, and where $\Sigma$ is the solution of the
matrix equation 
$\Sigma (2 N^{-1} \gamma_{\star} F(\theta_\star) - I_d )+
(2 N^{-1} \gamma_{\star} F(\theta_\star) - I_d ) \Sigma 
=2 \gamma_{\star}^2 N^{-2} F(\theta_\star)$. Solving this equation, we obtain
$\Sigma = \gamma_{\star}^2 N^{-2} F(\theta_\star) (2 \gamma_\star N^{-1} 
F(\theta_\star) - I_d)^{-1}$. Notice that 
$\Sigma - F(\theta_\star)^{-1} = F(\theta_\star)^{-1} (2 \gamma_\star N^{-1} 
F(\theta_\star) - I_d)^{-1} ( \gamma_\star N^{-1} F(\theta_\star) - I_d)^2
> 0$, which quantifies the departure from asymptotic efficiency of the non 
averaged algorithm. 

\subsection{Application to source localization}
\label{subsec:localisation}

The distributed algorithm described above is used here to localize a source by
a collection of $N = 40$ sensors. The unknown location of the source in the
plane is represented by a parameter $\theta_\star \in \bR^2$. The sensors are
located in the square $[0, 50] \times [0, 50]$ as shown by
Figure~\ref{fig:EusipcoNet}, and they receive scalar-valued signals from the
source ($m_i = 1$ for all $i$). It is assumed that the density of $\bs X_1 \in
\bR^N$ is $f_\star(\bs x) = \prod_{i=1}^N f_i(\theta_\star, x_i)$ where
$f_i(\theta_\star, \cdot) = {\mathcal N}( 1000 / | \theta_\star - r_i |^2,
10^{-2})$ where $r_i \in \bR^2$ is the location of Node $i$. The fitted model
is $f(\theta, \bs x) = \prod_{i=1}^N f_i(\theta, x_i)$ with $f_i(\theta, \cdot)
= {\mathcal N}( 1000/| \theta - r_i |^2, 10^{-2})$ (see
\cite{rabbat:nowak:2004} for a similar model). The model for matrices $W_n$ is
the pairwise gossip model described in Section \ref{sec:gossip}.  The step
sequence $\gamma_n$ is set to $10^{-3}/n^{0.7}$.
Note that in practice, setting adequately the step size in order to find the sought tradeoff between a short transient phase and a good asymptotic accuracy is known to
be sensitive to the statistical model of interest. 
Finally, the initial value
$\bth_0 \in \bR^{2N}$ is chosen at random under the uniform distribution on the
square $[0, 50] \times [0,50]$.

The convergence of the distributed algorithm to the consensus subspace is
illustrated in Figures~\ref{fig:EusipcoAverageError}. Figure~\ref{fig:histog}
represents the empirical distribution of the normalized estimation error
$\gamma_n^{-1/2}(\athn-\theta_\star)$ after $n=50\,000$ iterations, based on 180 Monte-Carlo runs
of the trajectory $\bthn$ initialized in the vincinity of $\theta_\star$. The empirical distribution
is coherent with the asymptotic Gaussian distribution given by Theorem~\ref{th:clt:drm}.


\appendix
\subsection{Notations}
For a positive deterministic sequence $(a_n)_{n\geq 1}$, the notation 
$x_n = o(a_n)$ refers to a deterministic $\bR^\ell$-valued sequence 
$(x_n)_{n \geq 1}$ such that $\lim_{n \to \infty} a_n^{-1}|x_n| =0$. For 
$p>0$, we denote the $L^p$-norm of a random vector $X$ by 
$\| X\|_p \eqdef \bE( |X|^p )^{1/p}$. The notation $X_n = o_{L^p}(a_n)$ refers
to a $\bR^\ell$-valued r.v. $(X_n)_{n \geq 1}$ such that 
$\lim_{n \to \infty} a_n^{-1}\|X_n\|_p =0$, while 
$X_n = \mathcal{O}_{L^p}(a_n)$ refers to a $\bR^\ell$-valued r.v. 
$(X_n)_{n \geq 1}$ such that $\limsup_{n} a_n^{-1}\|X_n\|_p < \infty$. 
Finally, $X_n = \mathcal{O}_{w.p.1.}(a_n)$ stands for any
$\bR^\ell$-valued r.v. $(X_n)_{n \geq 1}$ such that $\limsup_{n} a_n^{-1}|X_n|$
is finite almost-surely.

\subsection{Proof of Theorems~\ref{the:cv} and ~\ref{the:cv:vanish}}
\label{sec:proof:the:cv:vanishingrate}
We give the proof of Theorem~\ref{the:cv:vanish}; the proof of
Theorem~\ref{the:cv} is on the same lines and details are omitted. We first
prove the almost-sure convergence to zero of $(\othn)_{n \geq 1}$. The
assumption $\PP{\limsup_n |\thn| < \infty}=1$ implies $\PP{\bigcup_{M \in
    \mathbb{Z}_+} \{\sup_n |\bth_n| \leq M \}} =1$ and we only have to prove
that for any $M >0$, with probability one, $ \lim_n \othn \un_{\sup_n |\thn|
  \leq M} = 0$.  To that goal, we write for any $\delta >0$, $m \geq 1$,
\begin{align*}
  \bP\big\{&\sup_{n \geq m} |\othn|  \un_{\sup_n |\thn| \leq M} \geq \delta\big\}  \\  &\leq
  \frac{1}{\delta^2} \bE\left( \sup_{n \geq m} |\othn|² \un_{\sup_n |\thn| \leq
      M}\right)  \\
  & \leq \frac{1}{\delta^2} \sum_{n \geq m} 
      n^{-2\alpha} \sup_n \bE\left( n^{2\alpha} |\othn|² 
                 \un_{\sup_{k \leq n-1} |\thk| \leq M}\right) \ .
\end{align*}
Lemma~\ref{lem:cvg:othn} and Assumption~{\bf \ref{hyp:vanish}} imply that
$(\othn)_{n\geq 1}$ converges to zero w.p.1. on the set $\{ \sup_n |\thn| \leq
M \}$.
\begin{lemma}
  \label{lem:cvg:othn}
  Suppose Assumptions {\bf
    \ref{hyp:model}\ref{hyp:model:a}-\ref{hyp:model:c})}, {\bf \ref{hyp:step}},
  {\bf \ref{hyp:H}\ref{hyp:H:a})} and {\bf \ref{hyp:vanish}}. Then for any $M
  >0$,
\[
\sup_n n^{2\alpha} \bE\left( |\othn|^2 \, \un_{\sup_{k \leq n-1} |\thk| \leq
    M} \right) < \infty \ .
\]
\end{lemma}
\begin{proof} Fix $M>0$. Recalling that $(A\otimes B) (C\otimes D) = 
(AC) \otimes (BD)$, let 
  ${\mathcal W}_n = 
  (W_{n}^T \otimes I_d) \Jo (W_{n}\otimes I_d) = (W_{n}^T ( I - N^{-1}
  \un\un^T) W_{n} ) \otimes I_d$.  Since $\othn = \Jo (W_{n} \otimes I_d) (
  \othnmu + \gamma_{n} \bsY_{n} )$, we have by Assumptions~{\bf
    \ref{hyp:model}\ref{hyp:model:a}-\ref{hyp:model:c})}
\begin{align*} 
\bE\big[ &|\othn|^2 \vert \mathcal{F}_{n-1} \big]\\
&= 
\bE\left[ ( \othnmu + \gamma_{n} \Jo \bsY_{n} )^T {\mathcal W}_{n} 
( \othnmu + \gamma_{n} \bsY_{n} ) \, \vert \mathcal{F}_{n-1} \right] \\
&\leq \rho_n \bE \left[ |  \othnmu + \gamma_{n} \bsY_{n} |^2 
 \, \vert \mathcal{F}_{n-1} \right] \\
&\leq \rho_n \bigg( 
| \othnmu |^2 + \gamma_{n}^2 \int |\bs y|^2 \mu_{\thnmu}(d \bs y)  \\
& \quad \quad+ 2 \gamma_{n} | \othnmu | 
\big( \int |\bs y|^2 \mu_{\thnmu}(d \bs y) \big)^{1/2} \bigg) 
\end{align*} 
By Assumption~{\bf \ref{hyp:H}\ref{hyp:H:a})}, 
\[
\sup_n \int |\bs y|^2 \mu_{\thnmu}(d \bs y) \un_{\sup_{k \leq n} |\thk| \leq M} 
 < \infty \ .
\]
This implies that there exists a constant $C > 0$ such that 
\[
\bE\left[ |\othn|^2 \vert \mathcal{F}_{n-1} \right] \leq \rho_n |\othnmu |^2 +
\gamma_n^{2} C + 2 \gamma_n | \othnmu| \sqrt{C} \ .
\]
Therefore,
\begin{align*}
  \bE\big[ &|\othn|^2 \un_{\sup_{k \leq n-1} |\bs \theta_k| \leq M} \big]\\
 & \leq \rho_n \bE\left[|\othnmu |^2 \un_{\sup_{k \leq n-2} |\bs \theta_k| \leq
      M}
  \right] + \gamma_n^{2} C   \\
  &+ 2 \gamma_n \left(C \, \bE\left[|\othnmu |^2 \un_{\sup_{k \leq n-2} |\bs
        \theta_k| \leq M} \right] \right)^{1/2} \ .
\end{align*}
The proof now follows the same lines as in the proof of \cite[Lemma 1 (see Eq.
(17))]{bianchi:jakubowicz:VT:2011} (see also Lemma~\ref{lem:suites} below,
Eq.~(\ref{eq:majun}))
\end{proof}
\begin{remark}
  When Assumption~{\bf \ref{hyp:vanish}} is replaced with Assumption~{\bf
    \ref{hyp:model}\ref{hyp:model:iid})} and the condition $\lim_n
  \gamma_n/\gamma_{n-1} =1$, then for any $\bar \rho \in (\rho, 1)$ there
  exists a constant $C$ such that
\begin{multline*}
  \bE\left[ \gamma_n^{-2}|\othn|^2 \un_{\sup_{k \leq n-1} |\bs
      \theta_k| \leq M} \right] \leq \\ \bar \rho
  \bE\left[\gamma_{n-1}^{-2}|\othnmu |^2 \un_{\sup_{k \leq n-2} |\bs
      \theta_k| \leq M} \right] + C \ .
\end{multline*}
Therefore, Lemma~\ref{lem:cvg:othn} gets into
\[
\sup_n \gamma_n^{-2} \bE\left( |\othn|^2 \, \un_{\sup_{k \leq n-1} |\thk| \leq
    M} \right) < \infty \ ;
\]
(see also Theorem~\ref{the:moment2bounded} for a proof of this bound).
\end{remark}
Now, the study of the whole vector $\thn$ is reduced to the analysis of its
projection $J \thn= \un \otimes \athn$ onto the consensus space. We now focus
on the average $\athn$. The convergence of the sequence $(\athn)_{n \geq 1}$ is
a direct consequence of Lemma~\ref{lem:average} along 
with~\cite[Theorems 2.2. and 2.3.]{andrieu:2005}. 


\begin{lemma}
  \label{lem:average} 
  Under Assumptions {\bf \ref{hyp:model}a-b}), {\bf \ref{hyp:step}}, {\bf \ref{hyp:H}},
  {\bf\ref{hyp:vanish}} and Eq.~(\ref{eq:assump:stability}) it holds:
\[
\athn = \athnmu + \gamma_n h(\athnmu) + \gamma_n \zeta_n 
\]
with $\sup_n |\sum_{k=1}^n \gamma_k \zeta_k|< \infty$ with probability one. Then $\lim_n\sd (\athn, \mathcal{L}) =0$ with probability one. 
\end{lemma}
\begin{proof} Eqs.~(\ref{eq:algo}) and~(\ref{eq:average}) along with assumption {\bf
    \ref{hyp:model}\ref{hyp:model:a})} yield:
 \begin{equation}
   \label{eq:rm}
   \athn =  \athnmu + \gamma_n \la \bs Z_n\ra\ ,  
 \end{equation}
where $\bs Z_n \eqdef (W_n\otimes I_d)(\bsY_n+\gamma_n^{-1}\othnmu)$,
 upon noting that under Assumption {\bf \ref{hyp:model}\ref{hyp:model:a})},
 $(W_n \otimes I_d) J=J$. We write $\la \bs Z_n\ra = h(\athnmu)+e_{n}+\xi_{n}$
 where
\begin{eqnarray*}
  e_{n}&\eqdef& \la (W_n\otimes I_d)(\bsY_n+\gamma_n^{-1}\othnmu)\ra - \int \la \bs y \ra \mu_{\thnmu}(d \bs y)  \\
  \xi_{n} &\eqdef& \int \la \bs y \ra  \mu_{\thnmu}(d \bs y) - \int \la \bs y \ra \mu_{\un \otimes \la \thnmu \ra}(d \bs y) \ .
\end{eqnarray*}
By Assumption~{\bf \ref{hyp:H}\ref{hyp:H:abis})} and the inequality $2 a b \leq
a^2 + b^2$, for any $M>0$ there exists a constant $C$ such that
\begin{equation}
\bE\left|\un_{\sup_n |\thn| \leq M} \sum_{n\geq 1} \gamma_n \xi_n\right| 
\leq  C \  \left(\sum_{n\geq 1}\gamma_n^2 +\sum_{n\geq 1}\bE \left(\left|\othnmu\right|^2 \un_{\sup_n |\thn| \leq M} \right)\right)\ . \label{eq:xideuxsomme}
\end{equation}
Therefore, the RHS in (\ref{eq:xideuxsomme}) is finite under the condition~{\bf
  \ref{hyp:step}} and Lemma~\ref{lem:cvg:othn}, thus implying that $
\sum_{n\geq 1}\gamma_n \xi_n$ converges w.p.1. on the set $\{\sup_n |\thn| \leq
M \}$ for any $M>0$ and therefore w.p.1.  since $\PP{\sup_n |\thn| < \infty} =1$.

Since $\bE\left [e_n\,|\cF_{n-1}\right]=0$, the sequence $\left(S_n \eqdef
  \sum_{k=1}^n \gamma_k e_k \un_{\sup_{\ell \leq k-1} |\bs \theta_\ell| \leq M} \right)_{n\geq 1}$ is a martingale.  We prove that
it converges almost surely by estimating its second order moment. For any $k\geq 1$,
\begin{eqnarray*}
  \bE\left[ |S_k|^2\right] &\leq&  \sum_{n\geq 1}\gamma_n^2\, \bE\left[ \left|e_n\right|^2\un_{\sup_{\ell \leq n-1} |\bs \theta_\ell| \leq M} \right] \\
&\leq& \sum_{n\geq 1}\gamma_n^2\, \bE\left[ (\bY_n + \gamma_n^{-1}\othnmu)^TP_n (\bY_n + \gamma_n^{-1}\othnmu)\un_{\sup_{\ell \leq n-1} |\bs \theta_\ell| \leq M} \right]
\end{eqnarray*}
where we set $P_n:=N^{-2}W_n^T\un\un^TW_n\otimes I_d$. Note that $P_n$ is independent of $Y_n$ conditionally to $\cF_{n-1}$.
Since $W_n$ is a stochastic matrix, its spectral norm is bounded uniformly in~$n$. 
Therefore, there exists a constant $C>0$ such that:
\begin{align*}
  \bE\left[ |S_n|^2\right]& \leq C \sum_{n\geq 1}\gamma_n^2\, \bE\left[
    \left|\bY_n + \gamma_n^{-1}\othnmu\right|^2\un_{\sup_{\ell \leq n-1} |\bs \theta_\ell| \leq M} \right]  \\
  & \leq 2C \sum_{n\geq 1}\gamma_n^2\, \bE\left[ |\bY_n|^2\un_{\sup_{\ell \leq
        n-1} |\bs \theta_\ell| \leq M} \right] + 2C \sum_{n\geq 1}\bE\left[
    |\othnmu|^2\un_{\sup_{\ell \leq n-1} |\bs \theta_\ell| \leq M} \right] \ .
\end{align*}
By Assumption~{\bf \ref{hyp:H}\ref{hyp:H:a})},
\[
\sup_n \bE\left[ |\bY_n|^2\un_{\sup_{\ell \leq n-1} |\bs \theta_\ell| \leq M}
\right] < \infty \ .
\]
By Lemma~\ref{lem:cvg:othn}  and Assumption~{\bf
  \ref{hyp:step}} it follows that $\textsc{}\sup_n \bE\left[ |S_n|^2\right]$ is
finite thus implying that the martingale $(S_n)_{n \geq 1}$ converges almost
surely to a r.v. which is finite w.p.1. (see e.g.  \cite[Corollary
2.2.]{hall:heyde:1980}).  

We now consider the last term $\sum_k \gamma_k e_k \left(1 - \un_{\sup_{ \ell
      \leq k-1} | \bs \theta_\ell | \leq M} \right)$. On the set $\{\sup_n
|\thn| \leq M \}$, this sum is null.  This concludes the proof since
$\PP{\sup_n |\thn| < \infty} =1$.

\end{proof}

\subsection{Proof of Theorem~\ref{the:stab}}
\label{sec:proof:the:stab}
Our stability result relies on preliminary technical lemmas,
Lemmas~\ref{lem:suites} and \ref{lem:growth-nablaV}.  Theorem~\ref{the:stab} is
a consequence of Lemma~\ref{lem:agreement}: it is established that $\lim_n
\othn =0$ with probability one, which implies that $\PP{\limsup_n |\othn| <
  \infty}=1$. It is also established that $\PP{\limsup_n |\athn| <
  \infty}=1$. 
\begin{lemma}
\label{lem:suites}
Let $(\gamma_n)_{n\geq 0}$, $(\rho_n)_{n\geq 0}$ be respectively a positive    
and a $[0,1]$-valued sequence such that $\sum_n \gamma_n^2 < \infty$; and      
$u_n$, $v_n$ be two real sequences such that for $n \geq n_0$, 
\begin{eqnarray} 
u_n &\leq& \rho_n u_{n-1} + 
\gamma_n M\sqrt{u_{n-1}}(1 + u_{n-1}+v_{n-1})^{1/2}+
\gamma_n^2M\left(1 + u_{n-1}+v_{n-1} \right) \ , 
\label{eq:majun} \\ 
v_n &\leq& v_{n-1} + M u_{n-1} +
\gamma_n M\sqrt{u_{n-1}}\left(1+u_{n-1}+v_{n-1}\right)^{1/2}+
\gamma_n^2 M(1+u_{n-1}+v_{n-1})\ . 
\label{eq:majvn} 
\end{eqnarray}                        
Then: {\sl i)} $\sup_n v_n<\infty$, {\sl ii)} $ \limsup_n \phi_n u_n <
\infty$ for any positive sequence $(\phi_n)_{n \geq 0}$ such that
\begin{eqnarray}
  \label{eq:condition:phi:set1}
&& \limsup_n \left( \gamma_n \sqrt{\phi_n} + \frac{
    \phi_{n-1}}{\phi_n}\right) < \infty \ , \quad \liminf_n (\gamma_n \sqrt{\phi_n})^{-1}
\left( \frac{\phi_{n-1}}{\phi_n} -\rho_n\right) > 0 \ ,  \\
  \label{eq:condition:phi:set2}
&& \sum_n \phi_n^{-1} < \infty \ . 
\end{eqnarray}
\end{lemma}
\begin{proof}
  $\bullet$ Set $\tilde \gamma_n=(1+M)\gamma_n$. Define two sequences $(a_n,b_n)_{n\geq n_0}$ such that
  $a_{n_0}=b_{n_0}=\max(u_{n_0},v_{n_0})$ and for each $n\geq n_0+1$:
\begin{eqnarray}
a_n&=& \rho_n a_{n-1} + \tilde \gamma_n\sqrt{a_{n-1}} \, (1+ a_{n-1}+b_{n-1})^{1/2}+ \tilde \gamma_n^2(1+ a_{n-1}+b_{n-1})  \label{eq:a}\\
b_n&=& b_{n-1} +M a_{n-1} +\tilde \gamma_n \sqrt{a_{n-1}} (1+a_{n-1}+b_{n-1})^{1/2} +\tilde \gamma_n^2 (1+ a_{n-1} + b_{n-1})\ .
  \label{eq:vp}
\end{eqnarray}
It is straightforward to show by induction that $u_n\leq a_n$ and $v_n\leq b_n$
for any $n \geq n_0$.  In addition, $b_n = b_{n-1}+ a_n+(M-\rho_n)a_{n-1}$. Thus
for $n \geq n_0+1$,
$$
b_n = a_n + \sum_{k=n_0}^{n-1}(M+1-\rho_{k+1}) a_k\ .
$$
Define $A_n \eqdef (M+1)\sum_{k=n_0}^n a_k$, $n \geq n_0$.  The above equality
implies that $a_n\leq b_n\leq A_n$. As a consequence, Eq.~(\ref{eq:a}) implies:
\begin{eqnarray}
a_n&\leq& \rho_n a_{n-1} + \tilde \gamma_n\sqrt{a_{n-1}} \, (1+ 2\,A_{n-1})^{1/2}+ \tilde \gamma_n^2(1+ 2A_{n-1})\ .
\label{eq:maja}
\end{eqnarray}
As $(A_n)_{n \geq n_0}$ is a positive increasing sequence, for any $n \geq
n_0+1$,
\begin{equation}
\label{eq:majb}
  \frac{a_n}{A_n}\leq \rho_n \frac{a_{n-1}}{A_{n-1}} + \tilde \gamma_n\sqrt{\frac{a_{n-1}}{A_{n-1}}} \, \left(\frac{1}{A_{n_0}}+ 2\right)^{1/2}
+ \tilde \gamma_n^2\left(\frac{1}{A_{n_0}}+ 2\right)\ .
\end{equation}
 $\bullet$ Define $L^2 \eqdef 1/A_{n_0}+2$, 
 and $c_n \eqdef \phi_n {a_n}/{ A_n}$.  By (\ref{eq:majb}), for any $n \geq
 n_0+1$,
 \begin{equation}
\label{eq:cn}
 c_n\leq \rho_n  \frac{\phi_n}{\phi_{n-1}}  c_{n-1} +  L \tilde \gamma_n \sqrt{c_{n-1} \phi_n}  \sqrt{ \frac{\phi_n}{\phi_{n-1}} }  +  L^2\ \tilde \gamma_n^2 \phi_n,
\end{equation}
and under the assumption (\ref{eq:condition:phi:set1}), there exist $n_1 \geq
n_0$ and a constant $\xi >0$ such that for any $n \geq n_1$,
\begin{equation}
\sqrt{\frac{\phi_{n-1}}{\phi_n}}  L \xi \left\{1 + \xi L \tilde \gamma_n \sqrt{\phi_{n-1}} \right\}  \leq  \left( \frac{\phi_{n-1}}{\phi_n} - \rho_n \right) \left( \tilde \gamma_n \sqrt{\phi_n} \right)^{-1}\ .
\label{eq:lessthanone}
\end{equation}
Define 
\begin{equation}
  \label{eq:definition:A}
  A\eqdef\max\left(\frac{1}{\xi},\frac{1}{\xi^2},c_{n_1}\right) \ .
\end{equation}
We prove by induction on $n$ that $c_n \leq A$ for any $n \geq n_1$.  The claim
holds true for $n=n_1$ by definition of $A$.  Assume that $c_{n-1}\leq A$ for
some $n-1\geq n_1$. Using~(\ref{eq:cn}) and (\ref{eq:definition:A}), for $n
\geq n_1+1$,
$$
\frac{c_n}{A}\leq \rho_n \frac{\phi_n}{\phi_{n-1}} + \frac{L}{\sqrt{A}} \tilde \gamma_n
\sqrt{ \phi_n} \sqrt{ \frac{\phi_n}{\phi_{n-1}} } + \frac{L^2}{A} \ \tilde \gamma_n^2
\phi_n,
$$
By (\ref{eq:lessthanone}), the RHS is less than one so that $c_n \leq A$.
This proves
that $(c_n)_{n \geq n_0}$ is a bounded sequence.  

$\bullet$ We prove that $(A_n)_{n \geq n_0}$ is a bounded sequence.  Using the
fact that $\sup_{n \geq n_1} \rho_n\leq 1$, $(A_n)_{n \geq n_0}$ is increasing
and Eq.~(\ref{eq:maja}), it holds for $n \geq n_1+1$ 
\begin{eqnarray*}
A_n = A_{n-1} + a_n &\leq& A_{n-1} + a_{n-1} + \tilde \gamma_n\sqrt{a_{n-1}} \, \sqrt{A_{n-1}}L^{1/2}+ \tilde \gamma_n^2L^2 A_{n-1} \\
&\leq& \left(1 + c_{n-1} \phi_{n-1}^{-1} + L^{1/2} \tilde \gamma_n \phi_{n-1}^{-1/2} \sqrt{c_{n-1}} + \tilde \gamma_n^2L^2\right)  A_{n-1} .
\end{eqnarray*}
Finally, since $\sup_{n \geq n_1} c_{n}\leq A$ and $(1+t^2) \leq \exp(t^2)$,
there exists $C>0$ s.t.  for any $n\geq n_1+1$, $A_n \leq \exp\left( C
  \{\phi_{n-1}^{-1} + \tilde \gamma_n^2 \} \right) A_{n-1}$ (note that under
(\ref{eq:condition:phi:set1}), $\limsup_n \{\tilde \gamma_n / \sqrt{\phi_n} \} \phi_n
< \infty$). By assumptions, $\sum_n \{\phi_{n-1}^{-1} + \tilde \gamma_n^2 \}< \infty$,
$(A_n)_{n \geq n_0}$ is therefore bounded.

$\bullet$ The proof of the lemma is concluded upon noting that $v_n\leq b_n\leq
A_n$ and $ u_n\leq  a_n\leq \tilde \gamma_n^{2} c_n A_n$.

\end{proof}
\begin{remark}
  If the sequences $(\gamma_n, \rho_n)_{n\geq 0}$ are such that
\begin{eqnarray}
\label{eq:condition:phi:set1:rmk}
&& \limsup_n \left( \frac{
    \gamma_{n}}{\gamma_{n-1}} + \frac{
    1- \rho_{n-1}}{1-\rho_n}\right) < \infty \ , \ \  \liminf_n \frac{1}{1 - \rho_n}
\left( \frac{(1-\rho_{n-1})^2}{(1-\rho_n)^2}\frac{\gamma_n^2}{\gamma_{n-1}^2} -\rho_n\right) > 0 \ ,  \\
\label{eq:condition:phi:set2:rmk}
&& \sum_n \gamma_n^2 (1-\rho_n)^{-2} < \infty \ , 
\end{eqnarray}
then the conditions (\ref{eq:condition:phi:set1}) and
(\ref{eq:condition:phi:set2}) are satisfied with $\phi_n \eqdef (1-\rho_n)^2 /
\gamma_n^2$. Examples of sequences satisfying these conditions are $\rho_n =
1-a/n^\eta$, $\gamma_n = \gamma_0 / n^\xi$ with $0 \leq \eta < 1 \land
(\xi-1/2)$.
\end{remark}

\begin{lemma}\label{lem:growth-nablaV}  
  Let $V: \bR^d \to \bR^+$ be a differentiable function such that $\nabla V$ is
  Lipschitz on $\bR^d$. There exist  constants $C,C'$ such that for any $\theta \in \bR^d$,
  $|\nabla V(\theta)|^2 \leq C V(\theta) $, and for any $\theta,\theta'\in \bR^d$,
\begin{equation} 
V(\theta') \leq V(\theta) + \nabla V(\theta)^T(\theta'-\theta) 
+ C' |\theta'-\theta|^2 
\label{eq:V-taylor} 
\end{equation} 
\end{lemma}
\begin{proof}
Given any $\theta, \theta' \in
\bR^d$, we have
\[ 
V(\theta') = V(\theta) + 
\nabla V(\theta)^T(\theta'-\theta) + 
\int_0^1 \left( \nabla V(\theta + t(\theta' - \theta))
- \nabla V(\theta) \right)^T  (\theta' - \theta) \, dt . 
\]
This implies (\ref{eq:V-taylor}) since $\nabla V$ is Lipschitz. Then, Applying
(\ref{eq:V-taylor}) with $\theta' = \theta - \mu \nabla V(\theta)$ where $\mu >
0$ and recalling that $V$ is nonnegative, we also have $0 \leq V(\theta) - \mu
(1 - \mu C' ) |\nabla V(\theta) |^2$.  Choosing $\mu$ small enough, we
thus get the result.
\end{proof}

\begin{lemma}[Agreement and Stability]
\label{lem:agreement}
Suppose Assumptions {\bf \ref{hyp:model}a-b)}, {\bf \ref{hyp:step}}, {\bf
  \ref{hyp:V}\ref{hyp:V:a}-\ref{hyp:V:b}}) and {\bf\ref{hyp:vanish}}. Assume in
addition ST~{\bf \ref{ST:1}-\ref{ST:2})}. Then,
  \begin{enumerate}[a)]
  \item $\sum_{n\geq 1}\bE\left|\othn\right|^2<\infty$ and $(\othn)_{n \geq 1}$
    converges to zero w.p.1.
  \item $\sup_{n\geq 1} \bE V(\athn)<\infty$ and $\sup_{n} \bE\left[|
    \bsYn |^2 \right] < \infty$,
  \end{enumerate}
  where $\la \bs x \ra$ and $\bs x_\bot$ are given by (\ref{eq:average}) and
  (\ref{eq:orthogonal}).
\end{lemma}
\begin{proof}
Define $ u_n \eqdef \bE\left[|\othn|^2\right]$ and $ v_n
\eqdef\bE\left[V(\athn)\right]\ .$ We prove that there exists a constant $M>0$
and an integer $n_0$ such that for any $n\geq n_0$, 
inequalities~\eqref{eq:majun} and \eqref{eq:majvn} are satisfied. 
The proof is then concluded by application of Lemma~\ref{lem:suites} upon
noting that under assumption~{\bf \ref{hyp:step}}, the rate $\phi_n = n^{2
  \alpha}$ satisfies the conditions (\ref{eq:condition:phi:set1}) and
(\ref{eq:condition:phi:set2}).

{\em Proof of (\ref{eq:majun})}.  As $W_n\un=\un$, we have $\Jo(W_n\otimes
I_d)=\Jo(W_n\otimes I_d)\Jo$.  As a consequence, $\othn = \Jo(W_n\otimes
I_d)(\othnmu+\gamma_n \bsYn)$. We expand the square Euclidean norm of the
latter vector:
  $$
  |\othn|^2 = (\othnmu+\gamma_n
  \bsY_n)^T(\{W_n^T(I_N-\un\un^T/N)W_n\}\otimes I_d)(\othnmu+\gamma_n
  \bsY_n)\ .
  $$
  Integrate both sides of the above equation w.r.t. the r.v. $W_n$; by
  assumption~{\bf \ref{hyp:model}\ref{hyp:model:c})}
$$
\bE[|\othn|^2\,|\cF_{n-1},\bsY_n] \leq \rho_n |\othnmu+\gamma_n \bsY_n|^2\ .
$$
Under Assumption {\bf\ref{hyp:vanish}}, $\lim_n n (1-\rho_n) =+\infty$:
then, there exists $n_0$ such that $\rho_n<1$ for any $n\geq n_0$. We obtain:
$$
\bE[|\othn|^2] \leq \rho_n \bE[|\othnmu|^2]+2\gamma_n\bE[|\othnmu|\,
|\bsY_n|]+\gamma_n^2\bE[|\bsY_n|^2] \ ,
$$
for any $n\geq n_0$. From Cauchy-Schwartz inequality, $\bE[|\othnmu|\,
|\bsY_n|]\leq \sqrt{u_{n-1}}(\bE[|\bsY_n|^2])^{1/2}$. Thus,
$$
u_n \leq \rho_n u_{n-1}+2\gamma_n\sqrt{u_{n-1}}(\bE[|\bsY_n|^2])^{1/2}+\gamma_n^2\bE[|\bsY_n|^2] \ .
$$
By assumption~ST{\bf \ref{ST:2})}, we have the following estimate
$\bE[|\bsY_n|^2]\leq C_1\left(1 + v_{n-1} + u_{n-1} \right)$.  This completes
the proof of~(\ref{eq:majun}), for any constant $M$ larger than $1+C_1$.

{\em Proof of (\ref{eq:majvn})}.  Lemma~\ref{lem:growth-nablaV} is applied with
$\theta \leftarrow \athn$ and $\theta' \leftarrow \athnmu$.  We have to
evaluate the difference $\athn-\athnmu$. By~(\ref{eq:algo}),
$$
\athn = (\frac{\un^TW_n}N\otimes I_d)\left(\thnmu+\gamma_{n} \bsYn\right) \ .
$$
Therefore,
\begin{eqnarray}
\athn - \athnmu &=& \left(\frac{\un^TW_n-\un^T}N\otimes I_d\right)\thnmu + \left(\frac{\un^TW_n}N\otimes I_d\right)\gamma_n\bsYn\nonumber\\
 &=& \left(\frac{\un^TW_n-\un^T}N\otimes I_d\right)\othnmu + \left(\frac{\un^TW_n}N\otimes I_d\right)\gamma_n\bsYn\ ,
\label{eq:diff}
\end{eqnarray}
where the second equality is due to the fact that $W_n$ is row-stochastic.
Under Assumption~{\bf \ref{hyp:model}\ref{hyp:model:a})}, $\bE(W_n)$ is doubly
stochastic. Thus, using the assumption~{\bf \ref{hyp:model}\ref{hyp:model:c})}:
\begin{equation}
  \label{eq:moment1Dif}
  \bE[\athn - \athnmu|\cF_{n-1}] = \gamma_n  \int \la \bs y \ra \ \mu_{\thnmu}(d \bs y)  .
\end{equation}
Plugging~(\ref{eq:moment1Dif}) into~(\ref{eq:V-taylor}), there exists $C'$ such that for any $n$,
$$
 \bE[ V(\athn)|\cF_{n-1}] \leq V(\athnmu) +\gamma_n \nabla V(\athnmu)^T \int \la \bs y \ra \ \mu_{\thnmu}(d \bs y) + C' \bE[|\athn-\athnmu|^2|\cF_{n-1}] \ .
 $$
By the condition~{\bf \ref{hyp:V}\ref{hyp:V:b})}, the quantity $-\nabla
 V(\athnmu)^Th(\athnmu)$ is positive; therefore,
\begin{multline*}
  \bE[ V(\athn)|\cF_{n-1}] \leq V(\athnmu) +\gamma_n \nabla
  V(\athnmu)^T \left(\int \la \bs y \ra \ \mu_{\thnmu}(d \bs y)-h(\athnmu)\right) \\+ C'
  \bE[|\athn-\athnmu|^2|\cF_{n-1}] \ .
\end{multline*}
Using successively the conditions ST~{\bf \ref{ST:2})} and
Lemma~\ref{lem:growth-nablaV}, we have the estimate
\begin{eqnarray*}
\nabla V(\athnmu)^T \left(\int \la \bs y \ra \ \mu_{\thnmu}(d \bs y)-h(\athnmu)\right)&\leq& |\nabla V(\athnmu)|\, C_2|\othnmu| \\
&\leq& \sqrt{C}C_2 \sqrt{V(\athnmu)}\, |\othnmu|\ .
\end{eqnarray*}
Using Cauchy-Schwartz inequality, the expectation of
the above quantity is no larger than $\sqrt{C}C_2\sqrt{u_{n-1}v_{n-1}}$.
We obtain:
\begin{equation}
  \label{eq:taylorV2}
 v_n \leq v_{n-1} +\gamma_n \sqrt{C}C_2\sqrt{u_{n-1}(1+u_{n-1}+v_{n-1})}+ C' \bE[|\athn-\athnmu|^2] \ ,
\end{equation}
where we used the fact that $u_{n-1}\geq 0$. We now need to find an estimate for $\bE[|\athn-\athnmu|^2]$.
Using Minkowski's inequality on~(\ref{eq:diff}),
\begin{equation}
\label{eq:mink}
\bE[|\athn-\athnmu|^2]^{1/2} \leq \bE\left[\left| \left(\frac{\un^TW_n-\un^T}N\otimes I_d\right)\othnmu \right|^2\right]^{1/2} + \bE\left[\left|\left(\frac{\un^TW_n}N\otimes I_d\right)\gamma_n\bsYn \right|^2\right]^{1/2} 
\end{equation}
Focus on the first term of the RHS of the above inequality. Remark that
$$
\bE[(W_n^T\un-\un)(\un^TW_n-\un^T) |\cF_{n-1}] = \bE[W_n^T\un\un^TW_n]-\un\un^T\ ,
$$
where we used the assumption~{\bf \ref{hyp:model}\ref{hyp:model:c})} along
with the fact that $\bE(W_n)$ is doubly stochastic (see the condition~{\bf
  \ref{hyp:model}\ref{hyp:model:a})}).  Upon noting that the entries of
$W_n$ are in $[0,1]$ (as a consequence of assumption~{\bf
  \ref{hyp:model}\ref{hyp:model:a})}), the spectral norm of
$\bE[W_n^T\un\un^TW_n]-\un\un^T$ is bounded. Thus, there exists a constant $C'$
such that:
$$
\bE\left[\left| \left(\frac{\un^TW_n-\un^T}N\otimes I_d\right)\othnmu \right|^2\right] \leq C' u_{n-1}\ .
$$
By similar arguments, there exists a constant $C''$ such that 
\begin{eqnarray*}
\bE\left[\left|\left(\frac{\un^TW_n}N\otimes I_d\right)\gamma_n\bsYn \right|^2\right]  &\leq& C'' \gamma_n^2\, \bE|\bsY_n|^2 \\
&\leq& C_2 C'' \gamma_n^2\,\left(1 + u_{n-1} + v_{n-1} \right)
\end{eqnarray*}
where we used assumption ST{\bf \ref{ST:2})}. Putting this
together with~(\ref{eq:mink}),
\begin{eqnarray*}
\bE[|\athn-\athnmu|^2] &\leq& ( \sqrt{C'}\sqrt{u_{n-1}} + \gamma_n\sqrt{C_2 C''}\,\sqrt{1 + u_{n-1} + v_{n-1} })^2\\
 &\leq& C( u_{n-1} + \gamma_n^2\,(1 + u_{n-1} + v_{n-1}) + \gamma_n\sqrt{u_{n-1}(1 + u_{n-1} + v_{n-1})} )\ .
\end{eqnarray*}
where $C>0$ is some constant chosen large enough. Plugging the above inequality into~(\ref{eq:taylorV2}),
\begin{multline*}
 v_n \leq v_{n-1} + (C' C) u_{n-1} 
+ (\sqrt{C}C_2+  C' C) \gamma_n \sqrt{u_{n-1}(1+u_{n-1}+v_{n-1})}\\
+ C' C  \gamma_n^2\,(1 + u_{n-1} + v_{n-1})\ .
\end{multline*}
This proves that~(\ref{eq:majvn}) holds for any $M$ chosen large enough.

{\em Proof of $\sup_{n} \bE\left[| \bsYn |^2 \right] < \infty$.}  By
Assumptions~{\bf\ref{hyp:model}\ref{hyp:model:c})} and ST{\bf \ref{ST:2})}:
 \begin{equation}
   \label{eq:L2bound:Y}
   \bE \left[\left| \bs Y_n \right|^2 \right]=\bE \left[ \bE_{\bth_{n-1}} \left[\left| \bs Y \right|^2 \right]\right] \leq C_2 \left(1 +  \bE
   \left[V(\la \thnmu\ra)\right] + \bE
   \left[ \left|\othnmu\right|^2 \right] \right) \ .
 \end{equation}
 The proof follows since $\sup_n \bE
   \left[V(\la \thn\ra)\right] < \infty$ and $\bE \left[
   \left|\othn\right|^2 \right] \leq \sum_n \bE \left[ \left|\othn\right|^2
 \right] < \infty$.
\end{proof}

\begin{lemma}\label{lem:stability:athn}
  Suppose Assumptions {\bf \ref{hyp:model}a-b)}, {\bf \ref{hyp:step}}, {\bf
    \ref{hyp:V}\ref{hyp:V:a}-\ref{hyp:H:b}}) and {\bf\ref{hyp:vanish}}. Assume
  in addition ST{\bf \ref{ST:1}-\ref{ST:2})}.  Then, $\PP{\limsup_n
    |\athn| < \infty}=1$.
\end{lemma}
\begin{proof} The sequence $(\athn)_{n \geq 1}$ satisfies the equation~(\ref{eq:SAmarkovnoise}).
  The proof is an application of \cite[Theorem 2.2.]{andrieu:2005}: in order to
  apply this Theorem, we only have to prove that with probability one
  \textit{(i)} the sequence $(\athn)_{n \geq 1}$ is infinitely often in a level
  set $\{V \leq M \}$ i.e.  $\PP{\liminf_n V(\athn) < \infty}=1$ and
  \textit{(ii)}
\[
 \sum_n \gamma_n  \left( (W_n\otimes I_d)(\bsY_n+\gamma_n^{-1}\othnmu) - h(\athnmu) \right)< \infty \ .
\]
For the recurrence property, we have
\[
\bE\left( \liminf_n V(\athn) \right) \leq \liminf_n \bE\left( V(\athn) \right)
\leq \sup_n \bE\left( V(\athn) \right) \ .
\]
By Lemma~\ref{lem:agreement}, the RHS is finite thus showing that
$\PP{\liminf_n V(\athn) < \infty}=1$. For the second property, we write $\la
(W_n\otimes I_d)(\bsY_n+\gamma_n^{-1}\othnmu) \ra - h(\athnmu) =
e_{n}+\xi_{n-1}$ where
\begin{eqnarray*}
  e_{n}&\eqdef& \la (W_n\otimes I_d)(\bsY_n+\gamma_n^{-1}\othnmu)\ra - \int \la \bs y \ra \mu_{\thnmu}(d \bs y)  \\
  \xi_{n-1} &\eqdef& \int \la \bs y \ra  \mu_{\thnmu}(d \bs y) - \int \la \bs y \ra \mu_{\un \otimes \la \thnmu \ra}(d \bs y) \ .
\end{eqnarray*}
By Assumption ST{\bf \ref{ST:2})} and the inequality $2 a b \leq a^2 +
b^2$, there exists a constant $C$ such that
\begin{equation}
\bE\left|\sum_{n\geq 1} \gamma_n \xi_{n-1}\right| 
\leq  C \  \left(\sum_{n\geq 1}\gamma_n^2 +\sum_{n\geq 1}\bE\left|\othnmu\right|^2 \right)\ . \label{eq:xideuxsommeBis}
\end{equation}
Therefore, the RHS in (\ref{eq:xideuxsommeBis}) is finite under the condition~{\bf
  \ref{hyp:step}} and Lemma~\ref{lem:agreement}, thus implying that $
\sum_{n\geq 1}\gamma_n \xi_n$ converges w.p.1.
Since $\bE\left [e_n\,|\cF_{n-1}\right]=0$, the sequence $\left(S_n \eqdef
  \sum_{k=1}^n \gamma_k e_k\right)_{n\geq 1}$ is a martingale.  We prove that
it converges almost surely by estimating its second order moment. For any $k\geq 1$,
\begin{eqnarray*}
  \bE\left[ |S_k|^2\right] &\leq&  \sum_{n\geq 1}\gamma_n^2\, \bE\left[ \left|e_n\right|^2\right] \\
&\leq& \sum_{n\geq 1}\gamma_n^2\, \bE\left[ (\bY_n + \gamma_n^{-1}\othnmu)^TP_n (\bY_n + \gamma_n^{-1}\othnmu)\right]
\end{eqnarray*}
where we set $P_n:=N^{-2}W_n^T\un\un^TW_n\otimes I_d$. Note that $P_n$ is independent of $Y_n$ conditionally to $\cF_{n-1}$.
Since $W_n$ is a stochastic matrix, its spectral norm is bounded uniformly in $n$. 
Therefore, there exists a constant $C>0$ such that:
$$
\bE\left[ |S_n|^2\right] \leq C \sum_{n\geq 1}\gamma_n^2\, \bE\left[
  \left|\bY_n + \gamma_n^{-1}\othnmu\right|^2\right] \leq 2C \sum_{n\geq
  1}\gamma_n^2\, \bE\left[ |\bY_n|^2\right] + 2C \sum_{n\geq 1}\bE\left[
  |\othnmu|^2\right] \ .
$$
By Lemma~\ref{lem:agreement} and
Assumption~{\bf \ref{hyp:step}} it follows that $\textsc{}\sup_n \bE\left[
  |S_n|^2\right]$ is finite thus implying that the martingale $(S_n)_{n \geq 1}$
converges almost surely to a r.v. which is finite w.p.1. (see e.g.
\cite[Corollary 2.2.]{hall:heyde:1980}).  This concludes the proof.

\end{proof}

\subsection{Proof of Theorem~\ref{the:moment2bounded}}
\label{sec:proof:the:moment2bounded}
Set $V_n\eqdef (I_N-\bs 1\bs 1^T/N) W_n$ and for any $1\leq k\leq n$,
\begin{equation}
  \label{eq:product:Vn}
  \Phi_{n,k} := (V_n\otimes I_d)(V_{n-1}\otimes I_d)\cdots(V_k\otimes I_d) \ .
\end{equation}
Note that by Assumptions~{\bf
  \ref{hyp:model}\ref{hyp:model:c}-\ref{hyp:model:iid})},
  \begin{eqnarray}
\|\Phi_{n,k}X\|_2^2 &=& \bE[X^T\Phi_{n-1,k}^T(V_n^TV_n\otimes I_d)\Phi_{n-1,k}X]
= \bE[X^T\Phi_{n-1,k}^T\bE(V_n^TV_n\otimes I_d  )\Phi_{n-1,k}X] \nonumber \\
&\leq& \rho \  \bE[X^T\Phi_{n-1,k}^T\Phi_{n-1,k} X]\ = \rho \|\Phi_{n-1,k}X\|_2^2\ .\label{lem:product:Vn}
\end{eqnarray}
From (\ref{eq:algo}) and since $\Jo(W_n\otimes I_d) = \Jo(W_n\otimes I_d)\Jo=
(V_n\otimes I_d)\Jo$ by Assumption~{\bf \ref{hyp:model}\ref{hyp:model:a})}, it
holds for any $n \geq 1$, $\othn = (V_n\otimes I_d)(\othnmu + \gamma_n \Yon)$.
By induction,
\begin{equation}
\label{eq:JothnTmp}
\othn = \sum_{k=1}^n \gamma_k \Phi_{n,k} \Yok + \Phi_{n,1}\bs\theta_{\bot,0}
\end{equation}
where $\Phi_{n,k}$ is defined by (\ref{eq:product:Vn}).  By
\eqref{lem:product:Vn} and Assumption~{\bf \ref{hyp:model}\ref{hyp:model:iid})},
the second term in the RHS
of~(\ref{eq:JothnTmp}) is a $\cO_{L^2}(\rho^{n/2})$. 
We now consider the first term in the RHS of~(\ref{eq:JothnTmp}).  Using
Minkowski's inequality and Equation~(\ref{lem:product:Vn})
\begin{align*} 
\|  \sum_{k=1}^n \gamma_k \Phi_{n,k} \Yok    \un_{\sup_{\ell \leq n-1} |\bs \theta_\ell| \leq M}  \|_2 & \leq \sum_{k=1}^n \gamma_k \|\Phi_{n,k} \Yok \un_{\sup_{\ell \leq n-1} |\bs \theta_\ell| \leq M} \|_2  \\
&  \leq 
 \sum_{k=1}^n \gamma_k \sqrt{\rho}^{n-k+1}\|\Yok\un_{\sup_{\ell \leq k-1} |\bs \theta_\ell| \leq M}  \|_2 \ .
\end{align*} 
By~\cite[Result~178,pp.38]{pol-sze-24}, the RHS is upper bounded by
$\limsup_{n\to\infty} \|\Yon \un_{|\bs \theta_{n-1}| \leq M} \|_2 \rho
(1-\sqrt{\rho})^{-1}$.  Under Assumption~{\bf \ref{hyp:H}\ref{hyp:H:a})}, this upper bound is finite (the proof follows
the same lines as in the proof of Lemma~\ref{lem:average} and is omitted). This
concludes the proof.
\subsection{Proof of Theorem~\ref{th:clt:drm}}
\label{proof:th:clt:drm}
Assumption~\ref{hyp:step} implies that $\lim_n \rho^{n/2} \gamma_n^{-2} =0$.
Upon noting that 
\[
\PP{\bigcup_M \{\sup_n |\thn| \leq M \} \vert \lim_q \bs \theta_q = \un \otimes
  \theta_\star}=1 \ ,
\]
Theorem~\ref{the:moment2bounded} implies that the sequence of r.v.  $(
\gamma_n^{-1/2} \othn)_{n}$ converges in probability to zero under the
conditional probability $\PP{\cdot \vert \lim_q \bs \theta_q = \un \otimes
  \theta_\star}$. Since $\thn = \un\otimes \la \thn \ra + \othn$, it remains to
prove that the sequence of r.v.  $(\gamma_n^{-1/2}(\athn - \theta_\star))_{n
  \geq 0}$ converges in distribution to $Z$ under the conditional distribution
given the event $\{\lim_q \theta_q = \un \otimes \theta_\star\}$.  To that
goal, we write 
\[
\athn = \athnmu + \gamma_n h\left( \athnmu \right) +\gamma_n e_n + \gamma_n
\xi_n
\]
where $\xi_{n} \eqdef \int \la \bs y \ra  \mu_{\thnmu}(d \bs y) - \int \la \bs y \ra \mu_{\un \otimes \la \thnmu \ra}(d \bs y)$ and 
\[
e_{n}\eqdef \la (W_n\otimes I_d)(\bsY_n+\gamma_n^{-1}\othnmu)\ra
- \int \la \bs y \ra  \mu_{\thnmu}(d \bs y) = \la \bsY_n \ra - \int \la \bs y \ra  \mu_{\thnmu}(d \bs y) \ ,
\]
since $\un^T W_n = \un^T$.  We then check the conditions C1 to C4 of
\cite[Theorem 1]{fort:2011} (see also \cite[Theorem 1]{pelletier:1998}).  Under
the assumptions {\bf \ref{assum:clt:thetastar}} and {\bf
  \ref{assum:clt:stepsize}\ref{assum:clt:stepsize:a})}, the conditions C1 and
C4 of \cite[Theorem 1]{fort:2011} are satisfied. We now prove C2b: there exists
a constant $C$ such that 
\begin{align*}
  \bE\left[ |e_{n+1}|^{2 + \tau} \un_{|\thn - \un \otimes \theta_\star | \leq
      \delta}\right] & \leq C \ \bE\left[ | \int \la \bs y \ra \mu_{\thn}(d \bs
    y)|^{2 +\tau} \un_{|\thn - \un \otimes \theta_\star | \leq \delta} \right]
  + C \ \bE\left[ | \la \bsY_{n+1} \ra |^{2 +\tau} \un_{|\thn - \un \otimes
      \theta_\star | \leq \delta} \right]
  \\
  & \leq 2C \sup_{|\bth - \un \otimes \theta_\star | \leq \delta} \int |\la \bs
  y \ra |^{2 +\tau} \mu_{\bth}(d \bs y)
\end{align*}
and the RHS is finite under Assumption {\bf \ref{assum:clt:Y}}. For C2c, we
have
\begin{align*}
  \bE \left[ e_{n+1} e_{n+1}^T \vert \cF_{n}\right] & = \left\{ \int \la \bs y
    \ra \la \bs y \ra^T \mu_{\thn}(d \bs y) - \left(\int \la \bs y \ra
      \mu_{\thn}(d \bs y) \right) \left( \int \la \bs y \ra \mu_{\thn}(d \bs y)
    \right)^T \right\}.
\end{align*}
By Assumption {\bf \ref{assum:clt:Y}}, this term converges w.p.1 to $\Upsilon$
on the set $\{\lim_k \thk = \un \otimes \theta_\star \}$. This concludes the
proof of C2.

We now consider the condition C3 of \cite{fort:2011} with $r_n = \xi_n + e_n
\un_{|\thnmu - \un \otimes \theta_\star | > \delta}$: we prove that for any
$M>0$, $ \gamma_n^{-1/2} r_n \un_{\sup_k |\thk| \leq M}\un_{\lim_k \bth_k
    = \un \otimes \theta_\star} = \mathcal{O}_{w.p.1}
o_{L^1}(1)$.  By Assumption~{\bf \ref{hyp:H}\ref{hyp:H:abis})}, there exists a
constant $C$ such that
\[
\gamma_n^{-1/2} \bE\left[|\xi_n | \un_{\lim_k \bth_k
    = \un \otimes \theta_\star}\un_{\sup_k |\thk| \leq M}\right]  \leq C \ \left( \gamma_n^{-1}
  \bE\left[|\othn |^2 \un_{\sup_k |\thk| \leq M}\right] \right)^{1/2}
\]
and the RHS tends to zero as $n \to \infty$ by
Theorem~\ref{the:moment2bounded}. On the set $\{\lim_n \thn = \un \otimes
\theta_\star \}$, the r.v. $e_n \un_{|\thnmu - \un \otimes \theta_\star | >
  \delta}$ is null for all large $n$. This concludes the proof of the condition C3
of \cite{fort:2011}, and the proof of Theorem~\ref{th:clt:drm}.

\subsection{Proof of Theorem~\ref{th:clt:drm:aver}}
\label{proof:th:clt:drm:aver}
We preface the proof by a preliminary result, established by \cite[Theorem
2]{fort:2011} (see also \cite{delyon:2000} for a similar result obtained under
stronger assumptions).
\begin{theo}
  \label{theo:rappel:TCLaver}
  Let $(\gamma_n)_n$ be a deterministic positive sequence such that $
  \log(\gamma_k/\gamma_{k+1}) = o(\gamma_k)$ and satisfying Assumption {\bf
    \ref{assum:clt:stepsize}\ref{assum:clt:stepsize:c}-\ref{assum:clt:stepsize:b})}.
  Consider the random sequence $(u_n)_n$ given by 
\[
u_{n+1} = u_n + \gamma_{n+1} h(u_n) + \gamma_{n+1} e_{n+1} + \gamma_{n+1} \xi_{n+1} \ , \qquad u_0 \in \bR^d \ ,
\]
where 
\renewcommand{\labelenumi}{AVER\theenumi .}
\begin{enumerate} 
\item \label{hyp:AVER1} $u_{\star}$ is a zero of the mean field: $h(u_{\star})=0$. The
  mean field $h:\bR^d \to \bR^{d}$ is twice continuously differentiable (in a
  neighborhood of $u_\star$) and $\nabla h(u_{\star})$ is a Hurwitz matrix.
\item \label{hyp:AVER2} \begin{enumerate}[(i)]
 \item  \ $(e_{n})_{n\geq 1}$ is a $\mathcal{F}_n$-adapted
   martingale-increment sequence.
 \item For any $M>0$, there exist $\tau>0$ s.t.  $\sup_{k} \bE \left[|e_{k}
     |^{2+\tau}\un_{ \sup_{\ell \leq k-1}|u_{\ell}-u_{\star}|\leq
       M}\right]<\infty$.
 \item There exists a positive definite (random) matrix $U_\star$ such that on
   the set $\{\lim_q u_q = u_\star \}$, $\lim_k \bE
   \left[e_{k}e_{k}^{T}|\mathcal{F}_{k-1} \right] = U_\star $ almost-surely.
 \end{enumerate}
\item \label{hyp:AVER3} $(\xi_{n})_{n\geq 1}$ is a $\mathcal{F}_n$-adapted sequence s.t. 
 \begin{enumerate}[(i)] 
 \item $ \gamma_n^{-1/2} \ | \xi_{n} | \un_{\lim_{q} u_{q}=u_{\star}}
   \un_{\sup_n |u_n| \leq M} = \ \mathcal{O}_{w.p.1}(1) \mathcal{O}_{L^2}(1)$
   for any $M>0$.
 \item $n^{-1/2} \sum_{k=0}^n \xi_{k+1} \un_{\lim_{q} u_{q}= u_{\star}}$
   converges to zero in probability.
\end{enumerate}
\end{enumerate}
Then for any $t \in \bR^d$,
  \begin{multline*}
    \lim_n \bE\left[\un_{\lim_q u_q = \theta_\star} \ \exp\left( i \sqrt{n} \ 
        t^T \left( \frac{1}{n} \sum_{k=1}^n u_k - u_\star \right)\right) \right]  \\
    = \bE\left[\un_{\lim_q u_q = u_\star} \ \exp \left(- \frac{1}{2}t^T \nabla
        h(u_\star)^{-1} \ U_\star \ \nabla h(u_\star)^{-T} t \right)
    \right] \ .
  \end{multline*}
\end{theo}
{\em Proof of Theorem~\ref{th:clt:drm:aver}.}  By
Theorem~\ref{the:moment2bounded} and Assumption~{\bf
  \ref{assum:clt:stepsize}\ref{assum:clt:stepsize:b})}, $\sqrt{N}^{-1}
\sum_{n=1}^N \othn \un_{\sup_\ell |\bs \theta_\ell| \leq M}$ converges in $L^2$
to zero for any $M>0$. Since $\thn = \othn + \un \otimes \la \thn \ra$, we now
prove a CLT for the averaged sequence $N^{-1} \sum_{n=1}^N \la \thn \ra$. To
that goal, we check the assumptions AVER\ref{hyp:AVER1} to AVER\ref{hyp:AVER3}
of Theorem~\ref{theo:rappel:TCLaver} with $u_n = \la \thn \ra$; $e_n, \xi_n$
defined as in the proof of Theorem~\ref{th:clt:drm}.  AVER\ref{hyp:AVER1} and
AVER\ref{hyp:AVER2} can be proved along the same lines as in the proof of
Theorem~\ref{th:clt:drm}; details are omitted.  Finally, by Assumption~{\bf
  \ref{hyp:H}\ref{hyp:H:abis})} and Theorem~\ref{the:moment2bounded}, $ \bE\left[
  |\xi_n |² \un_{\lim_k \bth_k = \un \otimes \theta_\star} \un_{\sup_{\ell \leq
      n-1} |\bs \theta_\ell| \leq M} \right] = \cO(\gamma_n^2)$; and
\[
\ell^{-1/2} \ \sum_{n=1}^\ell \bE\left[|\xi_n | \un_{\lim_k \bth_k = \un
    \otimes \theta_\star} \un_{\sup_{\ell \leq
      n-1} |\bs \theta_\ell| \leq M}\right] \leq C \ \ell^{-1/2} \ \sum_{n=1}^\ell \gamma_n \ .
\]
The RHS tends to zero under Assumption~{\bf
  \ref{assum:clt:stepsize}\ref{assum:clt:stepsize:b})} thus showing
AVER\ref{hyp:AVER3}.

\bibliographystyle{IEEEbib} \bibliography{biblio}

\newpage

\begin{figure}[t]
  \centering
  \includegraphics[width=0.5\linewidth]{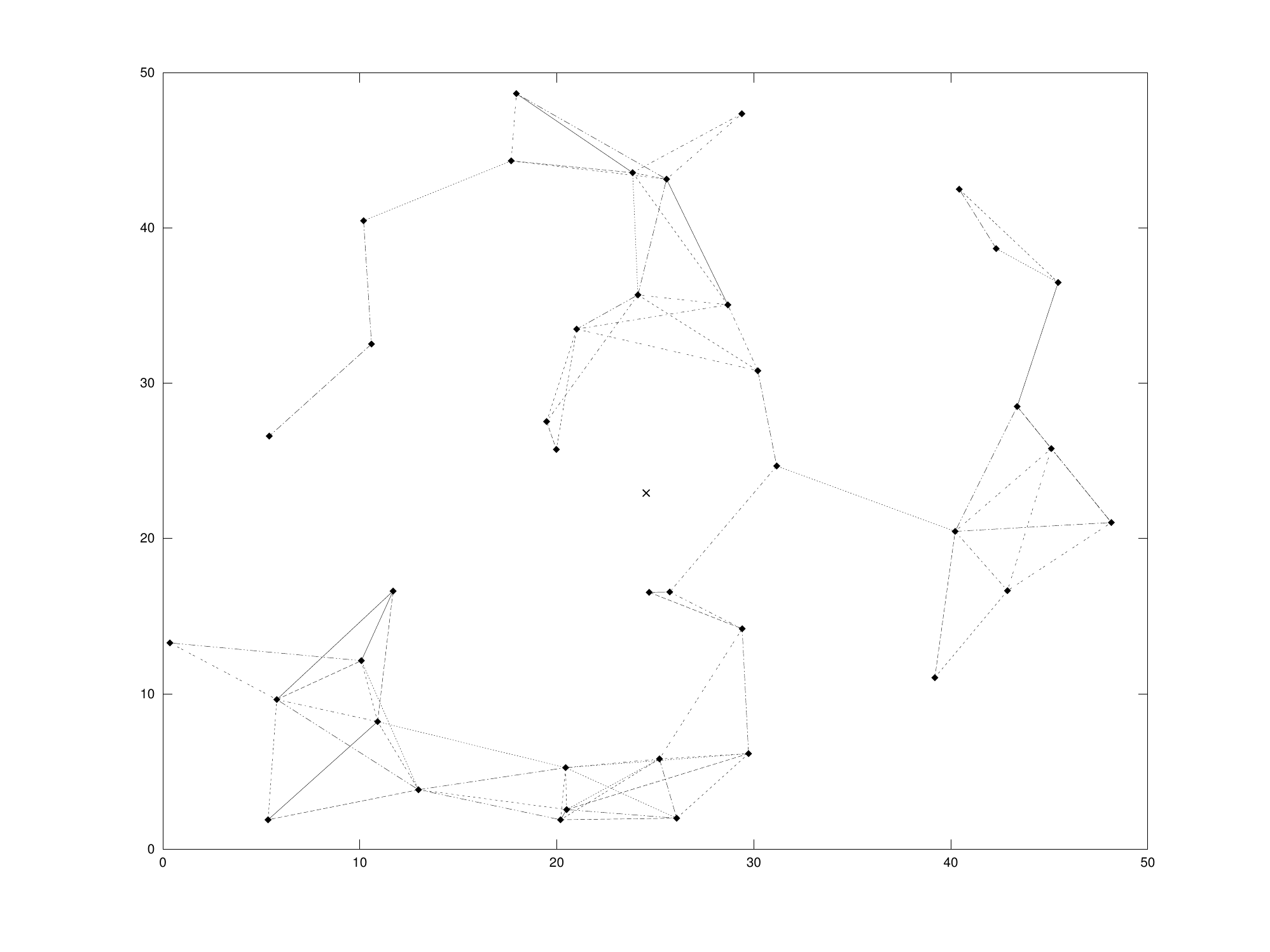}
  \caption{$N=40$ sensors with the graph (line segments) 
     and the source (star)}
  \label{fig:EusipcoNet}
\end{figure}

\begin{figure}[b]
  \centering
  \includegraphics[width=0.5\linewidth]{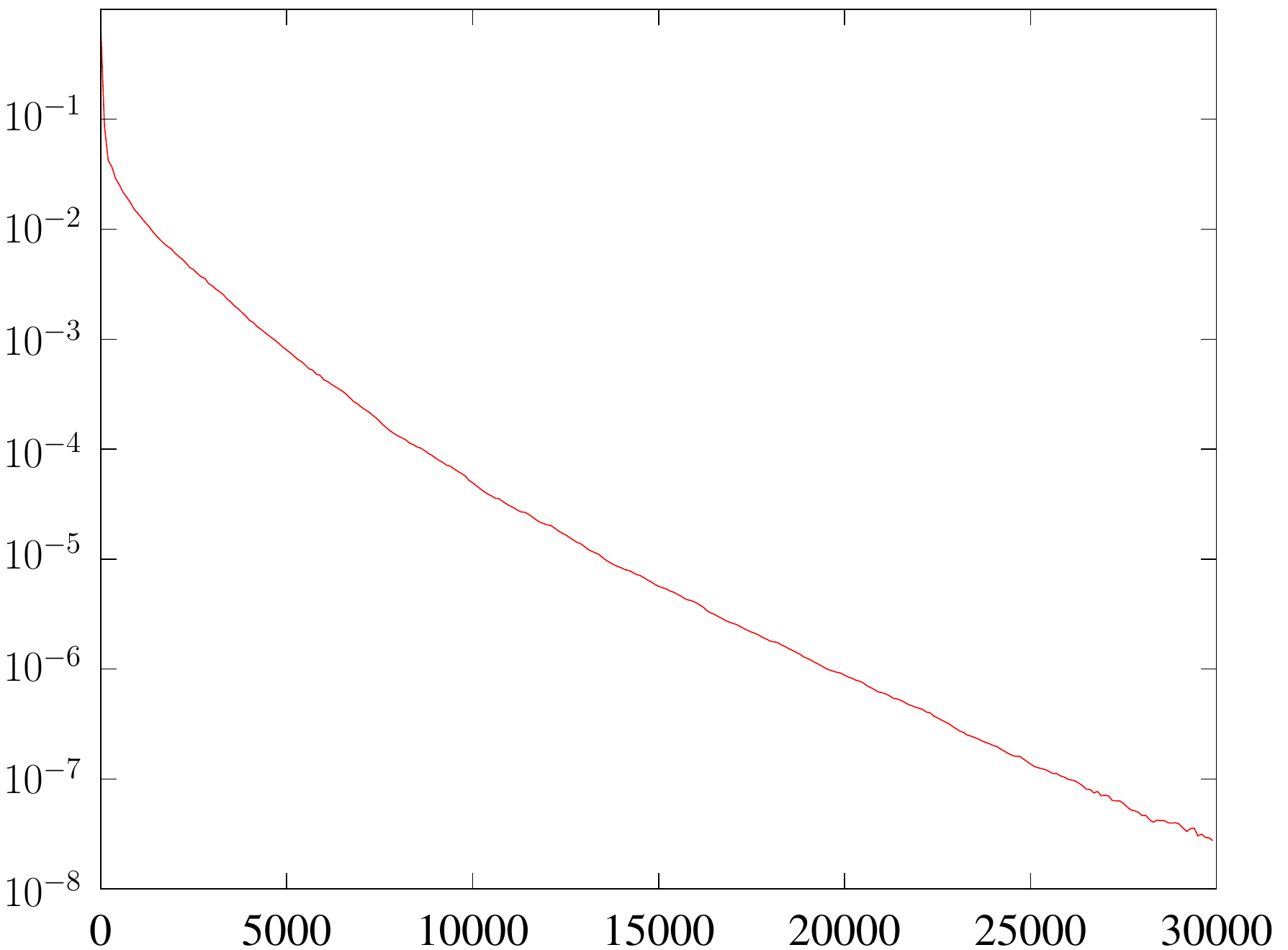}
  \caption{Square error per node $(1/N)\sum_i|\theta_{n,i}-\theta_\star|^2$ as a function of the number of iterations.}
  \label{fig:EusipcoAverageError}
\end{figure}

\newpage
\begin{figure}[h]
  \centering
  \includegraphics[width=0.5\linewidth]{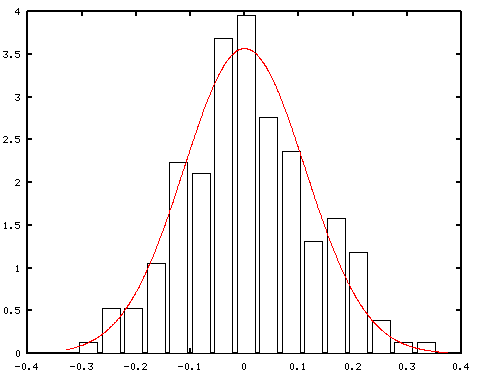}
  \caption{Empirical distribution of real part of the normalized estimation error
$\gamma_n^{-1/2}(\athn-\theta_\star)$ for $n=50\,000$ (bars) \emph{versus} asymptotic distribution given by Theorem~\ref{th:clt:drm} (solid line)}
  \label{fig:histog}
\end{figure}

\clearpage
\small

{\bf Pascal Bianchi} was born in 1977 in Nancy, France. He received the M.Sc. degree of the University of Paris XI and Supélec in 2000 and the Ph.D. degree of the University of Marne-la-Vallée in 2003. From 2003 to 2009, he was an Associate Professor at the Telecommunication Department of Supélec. In 2009, he joined the Statistics and Applications group at LTCI-Telecom ParisTech. His current research interests are in the area of distributed algorithms for multi-agent networks. They include distributed optimization, stochastic approximation, decentralized detection, quantization. 
\bigskip

{\bf Gersende Fort} was born in France, in 1974.  She received the Engineering
degree in Telecommunications from Ecole Nationale Supérieure des
Télécommunications (ENST), Paris, France, in 1997; the Masters degree from
Universit\'e Paris IV, France, in 1997; the PhD degree in Applied Mathematics
from Universit\'e Paris IV in 2001; and the Habilitation \`a Diriger les
Recherches from Universit\'e Dauphine-Paris IX
in 2010. \\
She joined the Centre national de la Recherche Scientifique (CNRS) in 2001 and
she is now a senior research scientist at Laboratoire Traitement et
Communication de l'Information (LTCI).  Her research interests are on Bayesian inverse problems and Monte Carlo methods. \\
She served as an associate editor for Bernoulli Journal since 2013.

\bigskip

{\bf Walid Hachem} was born in Bhamdoun, Lebanon, in 1967.  He received the
Engineering degree in telecommunications from St Joseph University (ESIB),
Beirut, Lebanon, in 1989, the Masters degree from Telecom ParisTech,  
France, in 1990, the PhD degree in signal
processing from Universit\'e Paris Est Marne-La-Vall\'ee in 2000 and
he Habilitation \`a Diriger des Recherches from Universit\'e Paris-Sud
in 2006. \\
After working in the telecommunications industry for ten years, 
he joined the academia in 2001 as a faculty member 
at Sup\'elec, France. In 2006, he joined the CNRS 
(Centre national de la Recherche Scientifique), where he is now a research
director at Telecom ParisTech. 
His research themes consist mainly in the random matrix theory and
its applications in signal processing, 
and in the decentralized estimation and optimization algorithms. \\ 
He served as an associate editor for the IEEE Transactions on Signal 
Processing between 2007 and 2010.

\end{document}